\documentclass[a4paper,
fontsize=11pt,%
oneside,%
numbers=enddot]{scrartcl}
\KOMAoptions{DIV=12} 
\usepackage[utf8]{inputenc}
\usepackage[T1]{fontenc}
\usepackage{textcomp}
\usepackage[fleqn]{amsmath}
\usepackage{amsthm}
\usepackage{amssymb}
\usepackage{amsfonts}
\usepackage{graphicx}
\usepackage{libertine}
\usepackage[libertine,
slantedGreek,
nosymbolsc,
nonewtxmathopt,
subscriptcorrection]{newtxmath}
\usepackage[scaled=0.95,varqu,varl]{inconsolata}
\usepackage[scr=boondox]{mathalpha} 
\usepackage{euscript} 
\usepackage{leftindex}
\allowdisplaybreaks[1] 
\numberwithin{equation}{section} 
\usepackage{xcolor}
\definecolor{dblue}{HTML}{0455BF}
\definecolor{dgreen}{HTML}{02724A}
\definecolor{dgreen2}{HTML}{025951}
\definecolor{dred}{HTML}{D90404}
\definecolor{dviolet}{HTML}{42208C}
\definecolor{labelkey}{HTML}{025951}
\definecolor{refkey}{HTML}{025951}
\definecolor{orng}{HTML}{D35400}
\definecolor{pblue}{rgb}{0.1176,0.5647,1}
\definecolor{pgreen}{rgb}{0.1961,0.8039,0.1961}
\definecolor{pred}{rgb}{1.0,0.2706,0.0}
\definecolor{fred}{rgb}{0.98,0.40,0.93}
\definecolor{pyellow}{rgb}{1.0,0.6471,0.0}
\usepackage{tikz}
\usepackage{pgfplots}
\usepgfplotslibrary{colormaps}

\pgfplotsset{colormap={setting1}{color=(pblue) color=(pyellow) 
color=(pred)},
colormap={setting2}{rgb255=(255,0,0) rgb255=(255,255,0)}}

\addtokomafont{section}{\centering}

\usepackage{enumitem}
\setlist{itemsep=-2.0pt}

\usepackage{upref} 
\usepackage{hyperref}
\hypersetup{colorlinks=true,
linktocpage=true,
linkcolor=dblue,
citecolor=dgreen,
urlcolor=dred,
pdfencoding=auto,
hypertexnames=false}
\makeatletter
\g@addto@macro\th@plain{
\thm@headfont{\bfseries\sffamily}
\thm@notefont{}}
\g@addto@macro\th@definition{
\thm@headfont{\bfseries\sffamily}
\thm@notefont{}}
\g@addto@macro\th@remark{
\thm@headfont{\bfseries\sffamily}
\thm@notefont{}}
\makeatother
\theoremstyle{plain}
\newtheorem{theorem}{Theorem}[section]
\newtheorem{proposition}[theorem]{Proposition}
\newtheorem{corollary}[theorem]{Corollary}
\newtheorem{lemma}[theorem]{Lemma}

\theoremstyle{definition}
\newtheorem{definition}[theorem]{Definition}
\newtheorem{example}[theorem]{Example}

\newtheorem{assumption}[theorem]{Assumption}
\theoremstyle{remark}
\newtheorem{remark}[theorem]{Remark}

\usepackage[extdef=true]{delimset}
\DeclareMathDelimiterSet{\scal}[2]{
\selectdelim[l]<{#1}
\mathpunct{}\selectdelim[p]|
{#2}\selectdelim[r]>}

\newcommand{\menge}[2]{\bigl\{{#1}\mid{#2}\bigr\}} 
\DeclareMathDelimiterSet{\Menge}[2]{\selectdelim[l]\{
{#1}\selectdelim[m]|{#2}\selectdelim[r]\}}
\newcommand*\Cdot{{\mkern 1.6mu\cdot\mkern 1.6mu}}
\makeatletter
\def\upintkern@{\mkern-7mu\mathchoice{\mkern-3.5mu}{}{}{}}
\def\upintdots@{\mathchoice{\mkern-4mu\@cdots\mkern-4mu}%
{{\cdotp}\mkern1.5mu{\cdotp}\mkern1.5mu{\cdotp}}%
{{\cdotp}\mkern1mu{\cdotp}\mkern1mu{\cdotp}}%
{{\cdotp}\mkern1mu{\cdotp}\mkern1mu{\cdotp}}}
\makeatother
\DeclareFontFamily{OMX}{mdbch}{}
\DeclareFontShape{OMX}{mdbch}{m}{n}{ <->s * [0.8] mdbchr7v }{}
\DeclareFontShape{OMX}{mdbch}{b}{n}{ <->s * [0.8] mdbchb7v }{}
\DeclareFontShape{OMX}{mdbch}{bx}{n}{<->ssub * mdbch/b/n}{}
\DeclareSymbolFont{uplargesymbols}{OMX}{mdbch}{m}{n}
\SetSymbolFont{uplargesymbols}{bold}{OMX}{mdbch}{b}{n}
\DeclareMathSymbol{\upintop}{\mathop}{uplargesymbols}{82}
\DeclareMathSymbol{\upointop}{\mathop}{uplargesymbols}{"48}
\makeatletter
\renewcommand{\int}{\DOTSI\upintop\ilimits@}
\renewcommand{\oint}{\DOTSI\upointop\ilimits@}
\makeatother

\newcommand{\qq}{\mathscr{Q}}
\newcommand{\RR}{\mathbb{R}}
\newcommand{\NN}{\mathbb{N}}
\newcommand{\HH}{\mathcal{H}}
\newcommand{\GG}{\mathcal{G}}
\newcommand{\KK}{\mathcal{K}}

\newcommand{\BL}{\ensuremath{\EuScript B}\,}
\newcommand{\BE}{\EuScript{B}}
\newcommand{\FF}{\EuScript{F}}

\newcommand{\GW}{\mathsf{G}_{\omega}}
\newcommand{\GK}{\mathsf{G}_{k}}
\newcommand{\GS}{\mathsf{G}}
\newcommand{\fw}{\mathsf{f}_{\omega}}
\newcommand{\gw}{\mathsf{g}_{\omega}}
\newcommand{\fs}{\mathsf{f}}
\newcommand{\gs}{\mathsf{g}}

\newcommand{\LW}{\mathsf{L}_{\omega}}
\newcommand{\LS}{\mathsf{L}}
\newcommand{\GH}{\mathfrak{H}}
\newcommand{\HS}{\mathsf{H}}
\newcommand{\pinf}{{+}\infty}
\newcommand{\minf}{{-}\infty}
\newcommand{\zeroun}{\intv[o]{0}{1}}
\newcommand{\rzeroun}{\intv[l]{0}{1}}

\newcommand{\RXX}{\intv{\minf}{\pinf}}
\newcommand{\RX}{\intv[l]0{\minf}{\pinf}}
\newcommand{\RP}{\intv[r]0{0}{\pinf}}
\newcommand{\RPP}{\intv[o]0{0}{\pinf}}

\newcommand{\emp}{\varnothing}

\newcommand{\infconv}{\mathbin{\mbox{\small$\square$}}}

\newcommand{\pushfwd}{\ensuremath{\mathbin{\raisebox{-0.20ex}
{\mbox{\LARGE$\triangleright$}}}}}
\newcommand{\epushfwd}{\ensuremath{\mathbin{\raisebox{-0.20ex}%
{\mbox{\LARGE$\trianglerightdot$}}}}}

\DeclareMathOperator{\Argmin}{Argmin}

\newcommand{\Id}{\mathrm{Id}}
\newcommand{\IdHS}{\mathsf{Id}_{\mathsf{H}}}

\newcommand{\moyo}[2]{\leftindex[I]^{#2}{#1}}

\DeclareMathOperator{\ran}{ran}

\DeclareMathOperator{\dom}{dom}

\DeclareMathOperator{\cam}{cam}

\DeclareMathOperator{\inte}{int}
\DeclareMathOperator{\reli}{ri}
\DeclareMathOperator{\sri}{sri}

\DeclareMathOperator{\rec}{rec}
\DeclareMathOperator{\prox}{prox}
\DeclareMathOperator{\proj}{proj}

\newcommand{\lev}[1]%
{{\ensuremath{{{{\operatorname{lev}}}_{\leq #1}}\,}}}
\DeclareMathOperator{\spc}{\overline{span}}

\newcommand{\PP}{\mathsf{P}}
\newcommand{\smalldot}{\ensuremath{\raisebox{0.12ex}{\scalebox{0.8}
{$\cdot$}}}}\def\trianglerightdot{{\mkern+1mu\smalldot\mkern-4.5mu 
\triangleright}}
\DeclareMathOperator{\PE}{\overset{\diamond}{\mathsf{E}}}
\DeclareFontFamily{U}{mathb}{}
\DeclareFontShape{U}{mathb}{m}{n}{<-5.5> mathb5 <5.5-6.5> mathb6 
<6.5-7.5> mathb7 <7.5-8.5> mathb8 <8.5-9.5> mathb9 <9.5-11> mathb10
<11-> mathb12}{}
\DeclareSymbolFont{mathb}{U}{mathb}{m}{n}
\DeclareFontSubstitution{U}{mathb}{m}{n}
\DeclareMathSymbol{\blackdiamond}{\mathbin}{mathb}{"0C}
\DeclareMathOperator{\pav}{{\mathsf{pav}}}

\DeclareMathOperator{\pcm}%
{\overset{\mbox{\,\large$\triangleright$}}{\mathsf{M}}}
\DeclareMathOperator{\epcm}%
{\overset{\mbox{\,\large$\trianglerightdot$}}{\mathsf{M}}}
\DeclareMathOperator{\pex}%
{\overset{\mbox{\,\large$\triangleright$}}{\mathsf{E}}}
\DeclareMathOperator{\epex}%
{\overset{\mbox{\,\large$\trianglerightdot$}}{\mathsf{E}}}
\DeclareMathOperator{\ccm}%
{\overset{\mbox{$\circ$}}{\mathsf{M}}}

\renewcommand{\leq}{\leqslant}
\renewcommand{\geq}{\geqslant}

\newcommand{\mae}{\text{\normalfont$\mu$-a.e.}}

\newcommand{\proxc}[1]{\mathbin%
{\ensuremath{\overset{#1}{\diamond}}}}
\newcommand{\proxcc}[1]{\mathbin%
{\ensuremath{\overset{#1}{\blackdiamond}}}}
\newcommand{\Rm}[1]{\overset{\mathord{\diamond}}{\mathsf{M}}_{#1}}
\newcommand{\Rcm}[1]{\overset{\mathord{\blackdiamond}}{\mathsf{M}}_{#1}}
\renewenvironment{abstract}{%
\vspace*{-0.50cm}
\small
\quotation%
\noindent%
{\normalfont\bfseries\sffamily
\nobreak\abstractname\ }%
}{%
\endquotation%
\medskip
}
\renewcommand{\abstractname}{Abstract.}

\newcommand\mscsname{MSC classification.}
\newenvironment{keywords}
{\renewcommand\abstractname{\keywordsname}\begin{abstract}}
{\end{abstract}}
\newenvironment{MSC}
{\renewcommand\abstractname{\mscsname}\begin{abstract}}
{\end{abstract}}
\usepackage[auth-sc]{authblk}
\newcommand{\email}[1]{\href{mailto:#1}{\nolinkurl{#1}}}
\renewcommand*\Affilfont{\normalfont\normalsize}
\newcommand\affilcr{\protect\\ \protect\Affilfont}
\makeatletter
\renewcommand\AB@affilsepx{\protect\\[0.5em]}
\makeatother

\author[1]{Patrick L. Combettes}
\affil[1]{North Carolina State University
\affilcr
Department of Mathematics
\affilcr
Raleigh, NC 27695, USA
\affilcr
\email{plc@math.ncsu.edu}
}
\author[2]{Diego J. Cornejo}
\affil[2]{North Carolina State University
\affilcr
Department of Mathematics
\affilcr
Raleigh, NC 27695, USA
\affilcr
\email{djcornej@ncsu.edu}
}

\begin{document}

\title{%
Variational Analysis of Proximal Compositions and Integral 
Proximal Mixtures\thanks{%
Contact author: P. L. Combettes.
Email: \email{plc@math.ncsu.edu}.
Phone: +1 919 515 2671. This work was supported by the National
Science Foundation under grant CCF-2211123.
}}

\date{~}
\maketitle

\vspace{12mm}

\begin{abstract} 
This paper establishes various variational properties of
parametrized versions of two convexity-preserving constructs that
were recently introduced in the literature: the proximal
composition of a function and a linear operator, and the integral
proximal mixture of arbitrary families of functions and linear
operators. We study in particular convexity, Legendre conjugacy,
differentiability, Moreau envelopes, coercivity, minimizers,
recession functions, and perspective functions of these constructs,
as well as their asymptotic behavior as the parameter varies. The
special case of the proximal expectation of a family of functions
is also discussed. 
\end{abstract}


\begin{MSC}
28B20,
46N10,
47N10,
49J53,
49N15.
\end{MSC}

\newpage

\section{Introduction}
\label{sec:1}

Throughout, $\HH$ is a real Hilbert space with power set $2^{\HH}$,
identity operator $\Id_{\HH}$, scalar product
$\scal{\Cdot}{\Cdot}_{\HH}$, associated norm $\norm{\Cdot}_{\HH}$,
and quadratic kernel $\qq_{\HH}=\norm{\Cdot}_{\HH}^2/2$. In
addition, $\GG$ is a real Hilbert space, the space of bounded
linear operators from $\HH$ to $\GG$ is denoted by $\BL(\HH,\GG)$,
and we set $\BL(\HH)=\BL(\HH,\HH)$. The Legendre conjugate of
$f\colon\HH\to\RXX$ is 
\begin{equation}
\label{e:jjm2}
f^*\colon\HH\to\RXX\colon x^*\mapsto\sup_{x\in\HH}
\brk1{\scal{x}{x^*}_{\HH}-f(x)},
\end{equation}
the Moreau envelope of index $\gamma\in\RPP$ of $f\colon\HH\to\RXX$
is
\begin{equation}
\label{e:jjm1}
\moyo{f}{\gamma}\colon\HH\to\RXX\colon x\mapsto
\inf_{y\in\HH}\brk2{f(y)+\dfrac{1}{\gamma}\qq_{\HH}(x-y)},
\end{equation}
and the adjoint of $L\in\BL(\HH,\GG)$ is denoted by $L^*$. 

In analysis, there are several ways to compose a function
$g\colon\GG\to\RXX$ and an operator $L\in\BL(\HH,\GG)$ in
order to construct a function from $\HH$ to $\RXX$. The most common
is the standard composition 
\begin{equation}
\label{e:1}
g\circ L\colon\HH\to\RXX\colon x\mapsto g(Lx). 
\end{equation}
Another instance is the infimal postcomposition of $g$ by $L^*$,
that is (see \cite[Section~12.5]{Livre1} and
\cite[Section~I.5]{Rock70}, and, for applications,
\cite{Beck14,Bric24,Xuef24}),
\begin{equation}
\label{e:2}
L^*\pushfwd g\colon\HH\to\RXX\colon x\mapsto
\inf_{\substack{y\in\GG\\ L^*y=x}}g(y).
\end{equation}
These two operations are dually related by the identities
$(L^*\pushfwd g)^*=g^*\circ L$ and, under certain qualification
conditions, $(g\circ L)^*=L^*\pushfwd g^*$ 
\cite[Corollary~15.28]{Livre1}. The focus of the present paper is
on the following alternative operations introduced in
\cite{Svva23}, where they were shown to manifest themselves in
various variational models.

\begin{definition}
\label{d:1}
Let $L\in\BL(\HH,\GG)$, $g\colon\GG\to\RXX$, and $\gamma\in\RPP$.
The \emph{proximal composition} of $g$ and $L$ with parameter
$\gamma$ is the function $L\proxc{\gamma}g\colon\HH\to\RXX$ given
by
\begin{equation}
L\proxc{\gamma}g=\brk2{\moyo{\brk1{g^*}}
{\frac{1}{\gamma}}\circ L}^*-\dfrac{1}{\gamma}\qq_{\HH},
\end{equation}
and the \emph{proximal cocomposition} of $g$ and $L$ with parameter
$\gamma$ is $L\proxcc{\gamma}g=(L\proxc{1/\gamma}g^*)^*$.
\end{definition}

In \cite{Svva23}, proximal compositions were studied only in the
case when $\gamma=1$ and few of their properties were explored. The
goal of this paper is to carry out an in-depth analysis of these
compositions, leading to results which are new even when
$\gamma=1$. We study in particular convexity, Legendre conjugacy,
differentiability, subdifferentiability, Moreau envelopes,
minimizers, recession functions, perspective functions,
as well as the preservation of properties such as coercivity,
supercoercivity, and Lipschitzianity. We also investigate the
behavior of $L\proxc{\gamma}g$ and $L\proxcc{\gamma}g$ as $\gamma$
varies. Another contribution of our work is to derive from these
results a systematic analysis of the notions of integral proximal
mixtures and comixtures. These operations, recently introduced in
\cite{Jota24}, combine arbitrary families of convex functions and
linear operators acting in different spaces in such a way that the
proximity operator of the mixture is explicitly computable in terms
of those of the individual functions. In turn, this analysis leads
to new results on the proximal expectation of a family of convex
functions.

The remainder of the paper is organized as follows. In
Section~\ref{sec:2}, we provide our notation and the necessary
mathematical background. In Section~\ref{sec:3}, we investigate
various variational properties of proximal compositions. Finally,
Section~\ref{sec:4} is devoted to applications to integral proximal
mixtures and proximal expectations.

\section{Notation and background}
\label{sec:2}

We first present our notation, which follows \cite{Livre1}
(see also the first paragraph of Section~\ref{sec:1}).

Let $L\in\BL(\HH,\GG)$. The range of $L$ is denoted by $\ran L$
and, if it is closed, the generalized inverse of $L$ is denoted by
$L^\dagger$. Further, $L$ is called an isometry if 
$L^*\circ L=\Id_{\HH}$ and a coisometry if $L\circ L^*=\Id_{\GG}$.
Let $f\colon\HH\to\RXX$. We set
\begin{equation}
\label{e:cam}
\begin{cases}
\cam f=\menge{h\colon\HH\to\RR}
{h~\text{is continuous, affine, and}~h\leq f}\\
\overline{f}=\sup\menge{h\colon\HH\to\RXX}
{h~\text{is lower semicontinuous and}~h\leq f}\\
\breve{f}=\sup\menge{h\colon\HH\to\RXX}
{h~\text{is lower semicontinuous, convex, and}~h\leq f}.
\end{cases}
\end{equation}
The infimal postcomposition of $f$ by $L\in\BL(\HH,\GG)$
(see \eqref{e:2}) 
is denoted by $L\epushfwd f$ if, for every $y\in L(\dom f)$,
there exists $x\in\HH$ such that $Lx=y$ and 
$(L\pushfwd f)(y)=f(x)\in\RX$. The function $f$ is proper if
$\dom f=\menge{x\in\HH}{f(x)<\pinf}\neq\emp$ and
$\minf\notin f(\HH)$. If $f$ is proper, its subdifferential is
\begin{equation}
\label{e:subdiff}
\partial f\colon\HH\to 2^{\HH}\colon x\mapsto\menge{x^*\in\HH}
{(\forall y\in\HH)\,\,\scal{y-x}{x^*}_{\HH}+f(x)\leq f(y)}
\end{equation}
and, if $f$ is also convex, its recession function at $x\in\HH$ is
\begin{equation}
(\rec f)(x)=\sup_{y\in\dom f}\brk1{f(x+y)-f(y)}.
\end{equation}
If $f$ and $g\colon\HH\to\RX$ are proper, their infimal convolution
is
\begin{equation}
\label{e:3}
f\infconv g\colon\HH\to\RXX\colon x\mapsto
\inf_{y\in\HH}\bigl(f(y)+g(x-y)\bigr).
\end{equation}
We denote by $\Gamma_0(\HH)$ the class of functions from $\HH$ to
$\RX$ which are proper, lower semicontinuous, and convex.
If $f\in\Gamma_0(\HH)$, its proximity operator is
\begin{equation}
\label{e:dprox}
\prox_{f}\colon\HH\to\HH\colon
x\mapsto\underset{y\in\HH}{\text{argmin}}\;
\brk1{f(y)+\qq_{\HH}(x-y)}.
\end{equation}
Let $C\subset\HH$. Then $\iota_C$ denotes the indicator function of
$C$ and $\sigma_C$ the support function of $C$. If $C$ is convex,
its normal cone is denoted by $N_C$ and its strong relative
interior is the set $\sri C$ of points $x\in C$ such that the
smallest cone containing $C-x$ is a closed vector subspace of
$\HH$. If $C$ is nonempty, closed, and convex, its projection
operator is denoted by $\proj_C$. Finally, the closed ball with
center $x\in\HH$ and radius $\rho\in\RPP$ is denoted by
$B(x;\rho)$. 

The following facts will be frequently used in the paper.

\begin{lemma}
\label{l:1}
Let $f$ and $g$ be functions from $\HH$ to $\RXX$. Then the
following hold:
\begin{enumerate}
\item
\label{l:1i}
$f^{**}\leq f$.
\item
\label{l:1ii}
$f\leq g\Rightarrow g^*\leq f^*$.
\item
\label{l:1iii}
$f^{***}=f^*$.
\item
\label{l:1v}
$f^*\equiv\pinf$ $\Leftrightarrow$ $\cam f=\emp$.
\item
\label{l:1iv}
$f^*\in\Gamma_0(\HH)$ $\Leftrightarrow$ 
$\mathrm{[}\,f$ is proper and 
$\cam f\neq\emp\,\mathrm{]}$.
\end{enumerate}
\end{lemma}
\begin{proof}
\ref{l:1i}--\ref{l:1iii}: \cite[Proposition~13.16]{Livre1}.

\ref{l:1v}: \cite[Proposition~13.12(ii)]{Livre1}.

\ref{l:1iv}: Combine \cite[Proposition~13.10(ii)]{Livre1} and
\ref{l:1v}.
\end{proof}

\begin{lemma}
\textup{\cite[Propositions~13.10(ii) and 13.23(i)--(ii)]{Livre1}}
\label{l:2}
Let $f\colon\HH\to\RXX$ and let $\rho\in\RPP$. Then the following
hold:
\begin{enumerate}
\item
\label{l:2i}
$(\rho f)^*=\rho f^*(\cdot/\rho)$.
\item
\label{l:2ii}
$(\rho f(\cdot/\rho))^*=\rho f^*$.
\item
\label{l:2iii}
$(f(\rho\cdot))^*=f^*(\cdot/\rho)$.
\end{enumerate}
\end{lemma}

The next lemma follows easily from \eqref{e:jjm1}.

\begin{lemma}
\label{l:3}
Let $f\colon\HH\to\RXX$, $\gamma\in\RPP$, and $\rho\in\RPP$.
Then the following hold:
\begin{enumerate}
\item
\label{l:3i}
$\rho(\moyo{f}{\gamma})=\moyo{(\rho f)}{\frac{\gamma}{\rho}}$.
\item
\label{l:3ii}
$(\moyo{f}{\gamma})(\rho\cdot)
=\moyo{(f(\rho\cdot))}{\frac{\gamma}{\rho^2}}$.
\end{enumerate}
\end{lemma}

\begin{lemma}
\label{l:8}
Let $f\in\Gamma_0(\HH)$ and $\gamma\in\RPP$. Then the following
hold:
\begin{enumerate}
\item
\label{l:8i--}
\textup{\cite[Theorem~9.20]{Livre1}}
$\cam f\neq\emp$.
\item
\label{l:8i-}
\textup{\cite[Corollary~13.38]{Livre1}}
$f^*\in\Gamma_0(\HH)$ and $f^{**}=f$.
\item
\label{l:8ii}
\textup{\cite[Corollary~16.30]{Livre1}}
$\partial f^*=(\partial f)^{-1}$.
\item
\label{l:8ii+}
\textup{\cite[Remark~14.4]{Livre1}}
$\moyo{f}{1}+\moyo{(f^{*})}{1}=\qq_{\HH}$ 
and $\prox_{f}+\prox_{f^*}=\Id_{\HH}$.
\item
\label{l:8vi}
\textup{\cite[Theorem~13.49]{Livre1}}
$\rec (f^*)=\sigma_{\dom f}$ and $\rec f=\sigma_{\dom f^*}$.
\item
\label{l:8iii}
\textup{\cite[Propositions~12.15 and 12.30]{Livre1}}
$\moyo{f}{\gamma}\colon\HH\to\RR$ is convex and Fr\'{e}chet
differentiable.
\item
\label{l:8v}
\textup{\cite[Proposition~12.30]{Livre1}}
$\nabla(\moyo{f}{\gamma})=(\Id_{\HH}-\prox_{\gamma f})/\gamma$.
\item
\label{l:8iv}
\textup{\cite[Proposition~14.1]{Livre1}}
$(f+\gamma\qq_{\HH})^*=\moyo{(f^*)}{\gamma}$.
\end{enumerate}
\end{lemma}

\begin{lemma}
\label{l:5}
Let $f\colon\HH\to\RX$, $L\in\BL(\HH,\GG)$, and $\gamma\in\RPP$.
Then the following hold:
\begin{enumerate}
\item
\label{l:5i}
\textup{\cite[Proposition~13.24(iii)]{Livre1}}
$(\moyo{f}{\gamma})^*=f^*+\gamma\qq_{\HH}$.
\item
\label{l:5ii}
\textup{\cite[Proposition~13.24(iv)]{Livre1}}
$(L\pushfwd f)^*=f^*\circ L^*$.
\item
\label{l:5iii}
\textup{\cite[Corollary~15.28(i)]{Livre1}}
Suppose that $f\in\Gamma_0(\HH)$ and $0\in\sri(\dom f-\ran L^*)$.
Then $(f\circ L^*)^*=L\epushfwd f^*$.
\end{enumerate}
\end{lemma}

\begin{lemma}
\label{l:6}
Let $f\in\Gamma_0(\HH)$, $g\in\Gamma_0(\HH)$, and 
$\gamma\in\RPP$ be such that 
$\moyo{f}{\gamma}=\moyo{g}{\gamma}$. Then $f=g$.
\end{lemma}
\begin{proof}
By Lemma~\ref{l:5}\ref{l:5i},
$f^*=(\moyo{f}{\gamma})^*-\gamma\qq_{\HH}
=(\moyo{g}{\gamma})^*-\gamma\qq_{\HH}=g^*$. Therefore, we deduce
from Lemma~\ref{l:8}\ref{l:8i-} that $f=f^{**}=g^{**}=g$.
\end{proof}

\begin{lemma}
\label{l:7}
Let $L\in\BL(\HH,\GG)$ and set $\Phi=\qq_{\GG}-\qq_{\HH}\circ L^*$.
Then $\Phi$ is convex if and only if $\|L\|\leq 1$.
\end{lemma}
\begin{proof}
Since $\dom\Phi=\GG$ and $\nabla\Phi=\Id_{\GG}-L\circ L^*$, we
deduce from \cite[Proposition~17.7]{Livre1} that $\Phi$ is convex
$\Leftrightarrow$ $\Id_{\GG}-L\circ L^*$ is monotone 
$\Leftrightarrow\|L^*\cdot\|_{\HH}^2\leq\|\cdot\|_{\GG}^2
\Leftrightarrow\|L^*\|\leq 1\Leftrightarrow\|L\|\leq 1$.
\end{proof}

\begin{lemma}
\label{l:10}
\textup{\cite[Proposition~17.36(iii)]{Livre1}}
Let $A\in\BL(\HH)$ be monotone and self-adjoint. Suppose that
$\ran A$ is closed, set $q_A\colon\HH\to\RR\colon
x\mapsto\scal{x}{Ax}_{\HH}/2$, and define $q_{A^\dagger}$
likewise. Then $q_{A}^*=\iota_{\ran A}+q_{A^\dagger}$.
\end{lemma}

\section{Proximal compositions}
\label{sec:3}

\subsection{General properties}

We start with direct consequences of Definition~\ref{d:1}.

\begin{proposition}
\label{p:1}
Let $L\in\BL(\HH,\GG)$, $g\colon\GG\to\RXX$, $\gamma\in\RPP$, and
$\rho\in\RPP$. Then the following hold:
\begin{enumerate}
\item
\label{p:1ii}
Let $h\colon\GG\to\RXX$ be such that $g^{**}\leq h\leq g$. Then
$L\proxc{\gamma}h=L\proxc{\gamma}g$ and
$L\proxcc{\gamma}h=L\proxcc{\gamma}g$.
\item
\label{p:1iii}
$(L\proxc{\gamma}g)^*=L\proxcc{1/\gamma}g^*$.
\item
\label{p:1iv}
$(L\proxcc{\gamma}g)^*=(L\proxc{1/\gamma}g^*)^{**}$.
\item
\label{p:1v}
$(L\proxc{\gamma}g)^{**}=(L\proxcc{1/\gamma}g^*)^*$.
\item
\label{p:1vi}
$\rho(L\proxc{\gamma}g)=L\proxc{\gamma/\rho}(\rho g)$.
\item
\label{p:1vii}
$(L\proxc{\gamma}g)(\rho\cdot)
=L\proxc{\gamma/\rho^2}(g(\rho\cdot))$.
\item
\label{p:1viii}
$\rho(L\proxcc{\gamma}g)=L\proxcc{\gamma/\rho}(\rho g)$.
\item
\label{p:1ix}
$(L\proxcc{\gamma}g)(\rho\cdot)
=L\proxcc{\gamma/\rho^2}(g(\rho\cdot))$.
\end{enumerate}
\end{proposition}
\begin{proof}
\ref{p:1ii}: By Lemma~\ref{l:1}\ref{l:1ii}--\ref{l:1iii}, 
$g^*=g^{***}\geq h^*\geq g^*$. Therefore, $h^*=g^*$, and the claims
follow from Definition~\ref{d:1}.

\ref{p:1iii}: It follows from Definition~\ref{d:1} and \ref{p:1ii}
that $L\proxcc{1/\gamma}g^*=(L\proxc{\gamma}g^{**})^*
=(L\proxc{\gamma}g)^*$.

\ref{p:1iv}: An immediate consequence of Definition~\ref{d:1}.

\ref{p:1v}: This follows from \ref{p:1iii}. 

\ref{p:1vi}: Combining Lemmas~\ref{l:2}\ref{l:2ii},
\ref{l:3}\ref{l:3i}--\ref{l:3ii}, and
\ref{l:2}\ref{l:2i}, we obtain
\begin{equation}
\label{e:p3i}
\rho\brk2{\moyo{\brk1{g^*}}{\frac{1}{\gamma}}\circ L}^*
=\brk2{\rho\moyo{\brk1{g^*}}{\frac{1}{\gamma}}\circ
\brk1{L/\rho}}^*
=\brk2{\moyo{\brk1{\rho g^*(\cdot/\rho)}}
{\frac{\rho}{\gamma}}\circ L}^*
=\brk2{\moyo{\brk1{(\rho g)^*}}{\frac{\rho}{\gamma}}\circ L}^*.
\end{equation}
The assertion therefore follows from
Definition~\ref{d:1}.

\ref{p:1vii}: We deduce from Lemmas~\ref{l:2}\ref{l:2iii} and
\ref{l:3}\ref{l:3ii} that
\begin{equation}
\label{e:p3ii}
\brk2{\moyo{\brk1{g^*}}{\frac{1}{\gamma}}\circ L}^*\brk1{\rho\cdot}
=\brk2{\moyo{\brk1{g^*}}{\frac{1}{\gamma}}\circ
\brk1{L/\rho}}^*
=\brk2{\moyo{\brk1{g^*(\cdot/\rho)}}
{\frac{\rho^2}{\gamma}}\circ L}^*
=\brk3{\moyo{\brk2{\brk1{g(\rho\cdot)}^*}}{\frac{\rho^2}{\gamma}}
\circ L}^*.
\end{equation}
In view of Definition~\ref{d:1}, the assertion is established.

\ref{p:1viii}: We invoke Definition~\ref{d:1},
Lemma~\ref{l:2}\ref{l:2ii}, \ref{p:1vi}, \ref{p:1vii}, and
Lemma~\ref{l:2}\ref{l:2i} to get
\begin{equation}
\rho\brk1{L\proxcc{\gamma}g}=\rho\brk1{L\proxc{1/\gamma}g^*}^*
=\brk2{\rho\brk1{L\proxc{1/\gamma}g^*}\brk1{\cdot/\rho}}^*
=\brk2{L\proxc{\rho/\gamma}\brk1{\rho g}^*}^*
=L\proxcc{\gamma/\rho}\brk1{\rho g}.
\end{equation}

\ref{p:1ix}: By Definition~\ref{d:1}, Lemma~\ref{l:2}\ref{l:2iii},
and \ref{p:1vii}, we get
\begin{align}
\brk1{L\proxcc{\gamma}g}\brk1{\rho\cdot}
=\brk1{L\proxc{1/\gamma}g^*}^*\brk1{\rho\cdot}
=\brk2{\brk1{L\proxc{1/\gamma}g^*}\brk1{\cdot/\rho}}^*
=\brk2{L\proxc{\rho^2/\gamma}\brk1{g(\rho\cdot)}^*}^*
=L\proxcc{\gamma/\rho^2}\brk1{g(\rho\cdot)},
\end{align}
which completes the proof.
\end{proof}

\begin{proposition}
\label{p:4}
Let $L\in\BL(\HH,\GG)$, let $g\colon\GG\to\RX$ be a proper function
such that $\cam g\neq\emp$, let $\gamma\in\RPP$, and set
$\Phi=\qq_{\GG}-\qq_{\HH}\circ L^*$. Then the following hold:
\begin{enumerate}
\item
\label{p:4i}
$L\proxc{\gamma}g=L^*\epushfwd(g^{**}+\Phi/\gamma)$.
\item
\label{p:4ii}
$L\proxcc{\gamma}g=(g^*+\gamma\Phi)^*\circ L$.
\item
\label{p:4iii}
$\dom(L\proxc{\gamma}g)=L^*(\dom g^{**})$.
\item
\label{p:4iii+}
Suppose that one of the following are satisfied: 
\begin{enumerate}
\item
\label{p:4iii+a}
$0<\|L\|<1$.
\item
\label{p:4iii+b}
$\dom g^{**}=\GG$.
\end{enumerate}
Then $\dom(L\proxcc{\gamma}g)=\HH$.
\item
\label{p:4iv}
$L\proxcc{\gamma}g\geq\moyo{(g^{**})}{\gamma}\circ L$.
\end{enumerate}
\end{proposition}
\begin{proof}
By Lemma~\ref{l:1}\ref{l:1iv}, $g^*\in\Gamma_0(\GG)$. Therefore,
Lemma~\ref{l:8}\ref{l:8iii} implies that
$\dom\moyo{(g^*)}{\frac{1}{\gamma}}=\GG$ and that
$\moyo{(g^*)}{\frac{1}{\gamma}}\in\Gamma_0(\GG)$.

\ref{p:4i}: Let $x\in\HH$. Because 
$\dom\moyo{(g^*)}{\frac{1}{\gamma}}-\ran L=\GG$, it follows from
Definition~\ref{d:1} and items \ref{l:5iii} and \ref{l:5i} in
Lemma~\ref{l:5} that
\begin{align}
\brk1{L\proxc{\gamma}g}(x)&
=\brk3{\brk2{\moyo{\brk1{g^*}}{\frac{1}{\gamma}}\circ L}^*
-\dfrac{1}{\gamma}\qq_{\HH}}(x)\nonumber\\
&=\brk3{L^*\epushfwd\brk2{\moyo{\brk1{g^*}}{\frac{1}{\gamma}}}^*}
(x)-\dfrac{1}{\gamma}\qq_{\HH}(x)\nonumber\\
&=\brk3{L^*\epushfwd\brk2{g^{**}+\dfrac{1}{\gamma}\qq_{\GG}}}(x)
-\dfrac{1}{\gamma}\qq_{\HH}(x)\nonumber\\
&=\min_{\substack{y\in\GG\\L^*y=x}}\brk2{g^{**}(y)
+\dfrac{1}{\gamma}\qq_{\GG}(y)}-\dfrac{1}{\gamma}\qq_{\HH}(x)
\nonumber\\
&=\min_{\substack{y\in\GG\\L^*y=x}}\brk2{g^{**}(y)
+\dfrac{1}{\gamma}\qq_{\GG}(y)-\dfrac{1}{\gamma}\qq_{\HH}(L^*y)}
\nonumber\\
&=\min_{\substack{y\in\GG\\L^*y=x}}\brk2{g^{**}(y)
+\dfrac{1}{\gamma}\Phi(y)}.
\end{align}

\ref{p:4ii}: By Definition~\ref{d:1}, \ref{p:4i}, and 
Lemmas~\ref{l:1}\ref{l:1iii} and \ref{l:5}\ref{l:5ii}, 
\begin{equation}
L\proxcc{\gamma}g=\brk1{L\proxc{1/\gamma}g^*}^*
=\brk2{L^*\epushfwd\brk1{g^{***}+\gamma\Phi}}^*
=\brk2{L^*\epushfwd\brk1{g^*+\gamma\Phi}}^*
=\brk1{g^*+\gamma\Phi}^*\circ L.
\end{equation}

\ref{p:4iii}: Since $\dom\Phi=\GG$,
\cite[Proposition~12.36(i)]{Livre1} and \ref{p:4i} imply that
$\dom(L\proxc{\gamma}g)=L^*(\dom(g^{**}+\Phi/\gamma))
=L^*(\dom g^{**})$.

\ref{p:4iii+}: By Lemma~\ref{l:7}, $\Phi\in\Gamma_0(\GG)$. Because
$\dom\Phi=\GG$, the identity $(\gamma\Phi)^*=\Phi^*/\gamma$ and
\cite[Proposition~15.2]{Livre1} imply that
\begin{equation}
\label{e:p4iii}
\brk1{g^*+\gamma\Phi}^*=g^{**}\infconv(\gamma\Phi)^*
=g^{**}\infconv\brk1{\Phi^*/\gamma}.
\end{equation}
On the other hand, we have $(1-\|L\|^2)\qq_{\GG}\leq\Phi$.
Hence, in view of property~\ref{p:4iii+a} and
Lemma~\ref{l:1}\ref{l:1ii}, we have 
$\Phi^*\leq\qq_{\GG}/(1-\|L\|^2)$, which yields
$\dom\Phi^*=\GG$. We thus deduce from \eqref{e:p4iii} that
$\dom(g^*+\gamma\Phi)^*=\dom g^{**}+\dom\Phi^*=\GG$ and obtain the
assertion via \ref{p:4ii}.

\ref{p:4iv}: Since $\Phi\leq\qq_{\GG}$,
$g^*+\gamma\Phi\leq g^*+\gamma\qq_{\GG}$. In turn,
Lemmas~\ref{l:8}\ref{l:8iv} and \ref{l:1}\ref{l:1ii}, and 
\ref{p:4ii} imply that
\begin{equation}
\moyo{\brk1{g^{**}}}{\gamma}\circ L
=\brk1{g^*+\gamma\qq_{\GG}}^*\circ L
\leq\brk1{g^*+\gamma\Phi}^*\circ L=L\proxcc{\gamma}g,
\end{equation}
which completes the proof.
\end{proof}

\begin{remark}
Suppose that $L\in\BL(\HH,\GG)$ satisfies $\|L\|=1$, set
$\Phi=\qq_{\GG}-\qq_{\HH}\circ L^*$, and set 
$A=\Id_{\GG}-L\circ L^*$. Then $A$ is monotone and self-adjoint, 
$\Phi\colon y\mapsto\scal{y}{Ay}_{\GG}/2$, and Lemma~\ref{l:10}
shows that $\dom\Phi^*=\ran A$ under the assumption that $\ran A$
is closed. In this case, arguing as in \eqref{e:p4iii} and using
Proposition~\ref{p:4}\ref{p:4ii}, we obtain
$\dom(L\proxcc{\gamma}g)=L^{-1}(\dom g^{**}+\ran A)$. 
\end{remark}

\begin{proposition}
\label{p:-1}
Let $L\in\BL(\HH,\GG)$ be such that $\ran L$ is closed and 
$\ker L=\{0\}$, let $g\colon\GG\to\RX$ be a proper function such
that $\cam g\neq\emp$, and let $\gamma\in\RPP$. Then the following
hold:
\begin{enumerate}
\item
\label{p:-1i}
Suppose that $g^{**}$ is coercive. Then $L\proxcc{\gamma}g$ is
coercive.
\item
\label{p:-1ii}
Suppose that $g^{**}$ is supercoercive. Then $L\proxcc{\gamma}g$ is
supercoercive.
\end{enumerate}
\end{proposition}
\begin{proof}
It follows from \cite[Fact~2.26]{Livre1} that there exists
$\alpha\in\RPP$ such that
$\|L\cdot\|_{\GG}\geq\alpha\|\cdot\|_{\HH}$. Thus, 
$\|Lx\|_{\GG}\to\pinf$ as $\|x\|_{\HH}\to\pinf$. On the other hand,
combining Lemmas~\ref{l:1}\ref{l:1iv} and \ref{l:8}\ref{l:8i-}, we
obtain $g^{**}\in\Gamma_0(\GG)$.

\ref{p:-1i}: By \cite[Corollary~14.18(i)]{Livre1},
$\moyo{(g^{**})}{\gamma}$ is coercive. Therefore,
Proposition~\ref{p:4}\ref{p:4iv} implies that
$(L\proxcc{\gamma}g)(x)\geq(\moyo{(g^{**})}{\gamma})(Lx)\to\pinf$
as $\|x\|_{\HH}\to\pinf$.

\ref{p:-1ii}: By \cite[Corollary~14.18(ii)]{Livre1},
$\moyo{(g^{**})}{\gamma}$ is supercoercive. Hence,
Proposition~\ref{p:4}\ref{p:4iv} yields
\begin{equation}
\dfrac{\brk1{L\proxcc{\gamma}g}(x)}{\|x\|_{\HH}}\geq
\dfrac{\moyo{\brk1{g^{**}}}{\gamma}(Lx)}{\|x\|_{\HH}}\geq
\alpha\dfrac{\moyo{\brk1{g^{**}}}{\gamma}(Lx)}{\|Lx\|_{\GG}}
\to\pinf\quad\text{as}\quad\|x\|_{\HH}\to\pinf,
\end{equation}
which concludes the proof.
\end{proof}

The next proposition studies the effect of quadratic perturbations
and translations.

\begin{proposition}
\label{p:5}
Let $L\in\BL(\HH,\GG)$, $g\in\Gamma_0(\GG)$, $\alpha\in\RR$,
$\gamma\in\RPP$, $\rho\in\RP$, and $u\in\HH$. Given $w\in\GG$, set
$\tau_w\,g\colon y\mapsto g(y-w)$. Then the following hold:
\begin{enumerate}
\item
\label{p:5i}
Set $\beta=\gamma/(1+\rho\gamma)$. Then 
$L\proxc{\gamma}(g+\rho\qq_{\GG}+\scal{\cdot}{Lu}_{\GG}+\alpha)
=(L\proxc{\beta}g)+\rho\qq_{\HH}+\scal{\cdot}{u}_{\HH}+\alpha$.
\item
\label{p:5ii}
$L\proxcc{\gamma}(\tau_{Lu}\,g+\alpha)
=\tau_{u}(L\proxcc{\gamma}g)+\alpha$.
\end{enumerate}
\end{proposition}
\begin{proof}
\ref{p:5i}: Let $x\in\HH$, set
$h=g+\rho\qq_{\GG}+\scal{\cdot}{Lu}_{\GG}+\alpha$, and set 
$\Phi=\qq_{\GG}-\qq_{\HH}\circ L^*$. Since $g\in\Gamma_0(\GG)$ and
$\rho\geq 0$, we have $h\in\Gamma_0(\GG)$. In turn,
Lemma~\ref{l:8}\ref{l:8i-} yields $h^*\in\Gamma_0(\GG)$,
$h^{**}=h$, and $g^{**}=g$. Therefore, it follows from
Proposition~\ref{p:4}\ref{p:4i} that
\begin{align}
\brk1{L\proxc{\gamma}h}(x)
&=\min_{\substack{y\in\GG\\L^*y=x}}\brk2{h(y)
+\dfrac{1}{\gamma}\Phi(y)}\nonumber\\
&=\min_{\substack{y\in\GG\\L^*y=x}}\brk2{g(y)+\rho\qq_{\GG}(y)
+\scal{y}{Lu}_{\GG}+\alpha+\dfrac{1}{\gamma}\Phi(y)}\nonumber\\
&=\min_{\substack{y\in\GG\\L^*y=x}}\brk2{g(y)+\rho\Phi(y)
+\rho\qq_{\HH}(L^*y)+\scal{L^*y}{u}_{\HH}+\dfrac{1}{\gamma}\Phi(y)}
+\alpha\nonumber\\
&=\min_{\substack{y\in\GG\\L^*y=x}}\brk2{g(y)
+\brk2{\rho+\dfrac{1}{\gamma}}\Phi(y)}
+\rho\qq_{\HH}(x)+\scal{x}{u}_{\HH}+\alpha\nonumber\\
&=\min_{\substack{y\in\GG\\L^*y=x}}\brk2{g(y)
+\dfrac{1}{\beta}\Phi(y)}+\rho\qq_{\HH}(x)+\scal{x}{u}_{\HH}
+\alpha\nonumber\\
&=\brk1{L\proxc{\beta}g}(x)+\rho\qq_{\HH}(x)+\scal{x}{u}_{\HH}
+\alpha.
\end{align}

\ref{p:5ii}: Set $h=\tau_{Lu\,}g+\alpha$. We recall from 
\cite[Proposition~13.23(iii)]{Livre1} that
$h^*=g^*+\scal{\cdot}{Lu}_{\GG}-\alpha$. Hence, using
Definition~\ref{d:1} and \ref{p:5i}, we get
\begin{align}
L\proxcc{\gamma}h
&=\brk2{L\proxc{1/\gamma}\brk1{g^*+\scal{\cdot}{Lu}_{\GG}
-\alpha}}^*\nonumber\\
&=\brk2{\brk1{L\proxc{1/\gamma}g^*}+\scal{\cdot}{u}_{\HH}-\alpha}^*
\nonumber\\
&=\tau_{u}\brk1{L\proxc{1/\gamma}g^*}^*+\alpha
\nonumber\\
&=\tau_{u}\brk1{L\proxcc{\gamma}g}+\alpha,
\end{align}
as claimed.
\end{proof}

\subsection{Convex-analytical properties}
\label{sec:3.2}

We first study the convexity, Legendre conjugacy, and
differentiability properties of proximal compositions. We then turn
our attention to the evaluation of their proximity operators,
subdifferentials, Moreau envelopes, recession functions, and
perspective functions. 

\begin{proposition}
\label{p:6}
Suppose that $0\neq L\in\BL(\HH,\GG)$, let $g\colon\GG\to\RX$ be a
proper function such that $\cam g\neq\emp$, let $\gamma\in\RPP$,
and let $\alpha\in\intv[r]{-1/\gamma}{\pinf}$. Suppose
that $g^{**}-\alpha\qq_{\GG}$ is convex and set
$\beta=(\alpha+1/\gamma)/\|L\|^2-1/\gamma$. Then
$L\proxc{\gamma}g-\beta\qq_{\HH}\in\Gamma_0(\HH)$.
\end{proposition}
\begin{proof}
By Lemma~\ref{l:1}\ref{l:1iv}, $g^*\in\Gamma_0(\GG)$. Thus,
Lemma~\ref{l:8}\ref{l:8iii} implies that
$\moyo{(g^*)}{\frac{1}{\gamma}}\circ L\in\Gamma_0(\HH)$.
In turn, Lemma~\ref{l:8}\ref{l:8i-} and Definition~\ref{d:1} yield
$L\proxc{\gamma}g+\qq_{\HH}/\gamma
=(\moyo{(g^*)}{\frac{1}{\gamma}}\circ L)^*\in\Gamma_0(\HH)$.
Since $(-\beta-1/\gamma)\qq_{\HH}$ is continuous with domain
$\GG$, by \cite[Lemma~1.27]{Livre1},
$L\proxc{\gamma}g-\beta\qq_{\HH}
=L\proxc{\gamma}g+\qq_{\HH}/\gamma
+(-\beta-1/\gamma)\qq_{\HH}$
is proper and lower semicontinuous. It remains to show that
$L\proxc{\gamma}g-\beta\qq_{\HH}$ is convex.
Let $x\in\HH$, set $\psi=\|L\|^2\qq_{\GG}-\qq_{\HH}\circ L^*$, and
set $\Phi=\qq_{\GG}-\qq_{\HH}\circ L^*$. By
Proposition~\ref{p:4}\ref{p:4i},
\begin{align}
\label{e:p15a}
\brk1{L\proxc{\gamma}g}(x)-\beta\qq_{\HH}(x)
&=\min_{\substack{y\in\GG\\L^*y=x}}\brk2{g^{**}(y)
+\dfrac{1}{\gamma}\Phi(y)}-\beta\qq_{\HH}(x)\nonumber\\
&=\min_{\substack{y\in\GG\\L^*y=x}}\brk2{g^{**}(y)
+\dfrac{1}{\gamma}\Phi(y)-\beta\qq_{\HH}(L^*y)}\nonumber\\
&=\min_{\substack{y\in\GG\\L^*y=x}}\brk2{g^{**}(y)
+\dfrac{1}{\gamma}\qq_{\GG}(y)
-\dfrac{1}{\|L\|^2}\brk2{\alpha+\dfrac{1}{\gamma}}\qq_{\HH}(L^*y)}
\nonumber\\
&=\min_{\substack{y\in\GG\\L^*y=x}}\brk2{\brk1{g^{**}(y)
-\alpha\qq_{\GG}(y)}+\brk2{\beta+\dfrac{1}{\gamma}}
\psi(y)}.
\end{align}
Since $\nabla\psi=\|L\|^2\Id_{\GG}-L\circ L^*$, for every
$y\in\GG$,
$\scal{\nabla\psi(y)}{y}_{\GG}=\|L\|^2\|y\|_{\GG}^2
-\|L^*y\|_{\HH}^2\geq 0$.
Therefore, we infer from \cite[Proposition~17.7]{Livre1} that
$\psi$ is convex.
Further, since $\alpha+1/\gamma\geq 0$, $(\beta+1/\gamma)\psi$ is
convex with domain $\GG$. By assumption,
$g^{**}-\alpha\qq_{\GG}\in\Gamma_0(\GG)$. Hence, the
function $(g^{**}-\alpha\qq_{\GG})+(\beta+1/\gamma)\psi$ is proper
and convex. Altogether, in view of \eqref{e:p15a} and
\cite[Proposition~12.36(ii)]{Livre1}, we conclude that
$L\proxc{\gamma}g-\beta\qq_{\HH}$ is convex.
\end{proof}

\begin{proposition}
\label{p:7}
Suppose that $L\in\BL(\HH,\GG)$ satisfies $0<\|L\|\leq 1$, let
$g\colon\GG\to\RX$ be a proper function such that $\cam g\neq\emp$,
and let $\gamma\in\RPP$. Then the following hold:
\begin{enumerate}
\item
\label{p:7i}
$L\proxc{\gamma}g\in\Gamma_0(\HH)$ and 
$L\proxcc{\gamma}g\in\Gamma_0(\HH)$. 
\item
\label{p:7ii}
$(L\proxcc{\gamma}g)^*=L\proxc{1/\gamma}g^*$.
\item
\label{p:7iii}
$L\proxc{\gamma}g=(L\proxcc{1/\gamma}g^*)^*$.
\end{enumerate}
\end{proposition}
\begin{proof}
Recall that Lemmas~\ref{l:1}\ref{l:1iv} and
\ref{l:8}\ref{l:8i--} assert that $g^*\in\Gamma_0(\GG)$ and
$\cam g^*\neq\emp$.

\ref{p:7i}:
Lemma~\ref{l:8}\ref{l:8i-} yields $g^{**}\in\Gamma_0(\GG)$. 
Now set $\beta=(1/\|L\|^2-1)/\gamma$. Then $\beta\geq 0$ and, 
by applying Proposition~\ref{p:6} with $\alpha=0$, we see
that $L\proxc{\gamma}g-\beta\qq_{\HH}\in\Gamma_0(\HH)$ and
hence that $L\proxc{\gamma}g\in\Gamma_0(\HH)$. Likewise, applying
Proposition~\ref{p:6} with $\alpha=0$ to $g^*\in\Gamma_0(\GG)$ 
and using Lemma~\ref{l:1}\ref{l:1iii} we get 
$L\proxc{1/\gamma}g^*\in\Gamma_0(\HH)$. In view of 
Definition~\ref{d:1} and Lemma~\ref{l:8}\ref{l:8i-}, we conclude
that $L\proxcc{\gamma}g\in\Gamma_0(\HH)$.

\ref{p:7ii}: 
We derive from Definition~\ref{d:1}, \ref{p:7i}, and
Lemma~\ref{l:8}\ref{l:8i-} that 
$(L\proxcc{\gamma}g)^*=(L\proxc{1/\gamma}g^*)^{**}
=L\proxc{1/\gamma}g^*$.

\ref{p:7iii}: By Proposition~\ref{p:1}\ref{p:1v},
\ref{p:7i}, and Lemma~\ref{l:8}\ref{l:8i-},
$(L\proxcc{1/\gamma}g^*)^*=(L\proxc{\gamma}g)^{**}
=L\proxc{\gamma}g$.
\end{proof}

The next result examines differentiability. 

\begin{proposition}
\label{p:18}
Suppose that $L\in\BL(\HH,\GG)$ satisfies $0<\|L\|\leq 1$, let
$g\colon\GG\to\RX$ be a proper function such that $\cam g\neq\emp$,
and let $\gamma\in\RPP$. Then the following hold:
\begin{enumerate}
\item
\label{p:18i}
Suppose that $\|L\|<1$ and set $\beta=\gamma(1/\|L\|^2-1)$. Then
$L\proxcc{\gamma}g$ is differentiable
with a $(1/\beta)$-Lipschitzian gradient.
\item
\label{p:18ii}
Let $\theta\in\RPP$, suppose that $g$ is real-valued, convex, and
differentiable with a $\theta$-Lipschitzian gradient, and set
$\beta=(1/\theta+\gamma)/\|L\|^2-\gamma$. Then
$L\proxcc{\gamma}g$ is differentiable with a
$(1/\beta)$-Lipschitzian gradient.
\end{enumerate}
\end{proposition}
\begin{proof}
We recall that a continuous convex function $f\colon\HH\to\RR$ is
differentiable with a $(1/\beta)$-Lipschitzian gradient if and only
if $f^*-\beta\qq_{\HH}$ is convex \cite[Theorem~18.15]{Livre1}.
Further, by Proposition~\ref{p:7}\ref{p:7ii},
$(L\proxcc{\gamma}g)^*=L\proxc{1/\gamma}g^*$.

\ref{p:18i}: By Proposition~\ref{p:4}\ref{p:4iii+a}, 
$\dom(L\proxcc{\gamma}g)=\HH$. Now set $\alpha=0$. Since
$\alpha>-\gamma$, we deduce from Proposition~\ref{p:6} that
$L\proxc{1/\gamma}g^*-\beta\qq_{\HH}$ is convex, i.e., that
$(L\proxcc{\gamma}g)^*-\beta\qq_{\HH}$ is convex.

\ref{p:18ii}: Since $g\in\Gamma_0(\GG)$, Lemma~\ref{l:8}\ref{l:8i-}
yields $\dom g^{**}=\dom g=\GG$. Thus, it results from 
Proposition~\ref{p:4}\ref{p:4iii+b} that
$\dom(L\proxcc{\gamma}g)=\HH$. Now set $\alpha=1/\theta$. Since 
$g^*-\alpha\qq_{\GG}$is convex and $\alpha>-\gamma$,
Proposition~\ref{p:6} implies that
$(L\proxcc{\gamma}g)^*-\beta\qq_{\HH}
=L\proxc{1/\gamma}g^*-\beta\qq_{\HH}$ is convex.
\end{proof}

\begin{remark}
Proposition~\ref{p:18}\ref{p:18i} guarantees the smoothness of 
the proximal cocomposition when $0<\|L\|<1$. 
Proposition~\ref{p:18}\ref{p:18ii} shows that the Lipschitz
constant of the gradient of the cocomposition is improved when the
original function is itself smooth. 
\end{remark}

The following proposition motivates calling $L\proxc{\gamma}g$ a
proximal composition.

\begin{proposition}
\label{p:17}
Suppose that $L\in\BL(\HH,\GG)$ satisfies $0<\|L\|\leq 1$, let
$g\colon\GG\to\RX$ be a proper function such that $\cam g\neq\emp$,
and let $\gamma\in\RPP$. Then the following hold:
\begin{enumerate}
\item
\label{p:17i}
$\prox_{\gamma\brk1{L\proxc{\gamma}g}}
=L^*\circ\prox_{\gamma g^{**}}\circ L$.
\item
\label{p:17ii}
$\prox_{\gamma\brk1{L\proxcc{\gamma}g}}=\Id_{\HH}-
L^*\circ\brk1{\Id_{\GG}-\prox_{\gamma g^{**}}}\circ L$.
\end{enumerate}
\end{proposition}
\begin{proof}
As previously noted, $g^*\in\Gamma_0(\GG)$ and 
$g^{**}\in\Gamma_0(\GG)$.

\ref{p:17i}:
It follows from Proposition~\ref{p:1}\ref{p:1vi} and
Definition~\ref{d:1} that
\begin{equation}
\label{e:p17}
\qq_{\HH}+\gamma\brk1{L\proxc{\gamma}g}
=\qq_{\HH}+L\proxc{1}(\gamma g)
=\brk2{\moyo{\brk1{(\gamma g)^*}}{1}\circ L}^*.
\end{equation}
Since Proposition~\ref{p:7}\ref{p:7i} yields
$L\proxc{\gamma}g\in\Gamma_0(\HH)$, we deduce from
\cite[Corollary~16.48(iii)]{Livre1}, \eqref{e:p17}, and
items \ref{l:8ii} and \ref{l:8v} in Lemma~\ref{l:8} that
\begin{equation}
\label{e:p17b}
\Id_{\HH}+\gamma\partial\brk1{L\proxc{\gamma}g}
=\partial\brk2{\qq_{\HH}+\gamma\brk1{L\proxc{\gamma}g}}
=\brk3{\nabla\brk2{\moyo{\brk1{(\gamma g)^*}}{1}\circ L}}^{-1}
=\brk2{L^*\circ\brk1{\Id_{\GG}-\prox_{(\gamma g)^*}}\circ L}^{-1}.
\end{equation}
Hence, by \cite[Proposition~16.44]{Livre1} and 
Lemma~\ref{l:8}\ref{l:8ii+},
$\prox_{\gamma\brk1{L\proxc{\gamma}g}}
=(\Id_{\HH}+\gamma\partial(L\proxc{\gamma}g))^{-1}
=L^*\circ\prox_{(\gamma g)^{**}}\circ L=
L^*\circ\prox_{\gamma g^{**}}\circ L$.

\ref{p:17ii}: By Proposition~\ref{p:1}\ref{p:1viii} and
Definition~\ref{d:1},
$\gamma(L\proxcc{\gamma}g)=L\proxcc{1}(\gamma g)
=(L\proxc{1}(\gamma g)^*)^*$. Therefore, 
Proposition~\ref{p:7}\ref{p:7i} and Lemma~\ref{l:8}\ref{l:8i-} 
entail that $(\gamma(L\proxcc{\gamma}g))^*=L\proxc{1}(\gamma g)^*$.
In turn, we deduce from Lemma~\ref{l:8}\ref{l:8ii+} and \ref{p:17i}
that $\prox_{\gamma\brk1{L\proxcc{\gamma}g}}
=\Id_{\HH}-\prox_{L\proxc{1}(\gamma g)^*}
=\Id_{\HH}-L^*\circ(\Id_{\GG}-\prox_{\gamma g^{**}})\circ L$.
\end{proof}

Our next result concerns the subdifferential of proximal
compositions. We recall that the parallel composition of
$A\colon\HH\to2^\HH$ by $L\in\BL(\HH,\GG)$ is 
$L\pushfwd A=(L\circ A^{-1}\circ L^*)^{-1}$
\cite[Section~25.6]{Livre1}.

\begin{proposition}
\label{p:9}
Suppose that $L\in\BL(\HH,\GG)$ satisfies $0<\|L\|\leq 1$, let
$g\colon\GG\to\RX$ be a proper function such that $\cam g\neq\emp$,
and let $\gamma\in\RPP$. Then the following hold:
\begin{enumerate}
\item
\label{p:9i}
$\partial(L\proxc{\gamma}g)=L^*\pushfwd(\partial g^{**}
+(\Id_{\GG}-L\circ L^*)/\gamma)$.
\item
\label{p:9ii}
$\partial(L\proxcc{\gamma}g)=L^*\circ(\partial g^*
+\gamma(\Id_{\GG}-L\circ L^*))^{-1}\circ L$.
\end{enumerate}
\end{proposition}
\begin{proof}
As seen in Proposition~\ref{p:7}\ref{p:7i}, 
$L\proxc{\gamma}g\in\Gamma_0(\HH)$ and 
$L\proxcc{\gamma}g\in\Gamma_0(\HH)$. Now set
$\Phi=\qq_{\GG}-\qq_{\HH}\circ L^*$ and $h=g^{**}+\Phi/\gamma$. We
deduce from Lemmas~\ref{l:1}\ref{l:1iv}, \ref{l:8}\ref{l:8i-},
and \ref{l:7} that $g^*\in\Gamma_0(\GG)$, 
$g^{**}\in\Gamma_0(\GG)$, and $\Phi\in\Gamma_0(\GG)$. Therefore,
since $\dom\Phi=\GG$, we have $h\in\Gamma_0(\GG)$ and, by
Lemma~\ref{l:8}\ref{l:8i-}, $h^{**}=h$. On the other hand, 
$\dom h^*\cap\ran L\neq\emp$ since 
Propositions~\ref{p:4}\ref{p:4ii} and \ref{p:7}\ref{p:7i}
yield $h^*\circ L=L\proxcc{1/\gamma}g^*\in\Gamma_0(\GG)$. Upon
invoking Propositions~\ref{p:4}\ref{p:4i} and
\ref{p:7}\ref{p:7iii}, we get
\begin{equation}
\label{e:p9a}
L^*\epushfwd h=L\proxc{\gamma}g=\brk2{L\proxcc{1/\gamma}g^*}^*
=\brk1{h^*\circ L}^*.
\end{equation}
Therefore, \cite[Proposition~16.42]{Livre1},
Lemma~\ref{l:8}\ref{l:8ii}, and 
\cite[Corollary~16.48(iii)]{Livre1} imply that
\begin{equation}
\label{e:p9b}
\partial(h^*\circ L)=L^*\circ\partial h^*\circ L
=L^*\circ\brk1{\partial h}^{-1}\circ L
=L^*\circ\brk1{\partial g^{**}+\nabla\Phi/\gamma}^{-1}\circ L.
\end{equation}

\ref{p:9i}: Combining \eqref{e:p9a}, Lemma~\ref{l:8}\ref{l:8ii},
and \eqref{e:p9b}, we obtain
\begin{equation}
\partial\brk1{L\proxc{\gamma}g}=\partial\brk1{h^*\circ L}^*
=\brk1{\partial(h^*\circ L)}^{-1}
=\brk2{L^*\circ\brk1{\partial g^{**}
+\nabla\Phi/\gamma}^{-1}\circ L}^{-1}
=L^*\pushfwd\brk1{\partial g^{**}+\nabla\Phi/\gamma}.
\end{equation}

\ref{p:9ii}: By Definition~\ref{d:1}, Lemma~\ref{l:8}\ref{l:8ii},
\ref{p:9i}, and Lemma~\ref{l:1}\ref{l:1iii}, 
\begin{equation}
\partial\brk1{L\proxcc{\gamma}g}
=\partial\brk1{L\proxc{1/\gamma}g^*}^*
=\brk2{\partial\brk1{L\proxc{1/\gamma}g^*}}^{-1}
=\brk2{L^*\pushfwd\brk1{\partial g^{***}+\gamma\nabla\Phi}}^{-1}
=L^*\circ\brk1{\partial g^*+\gamma\nabla\Phi}^{-1}\circ L,
\end{equation}
which completes the proof.
\end{proof}

\begin{corollary}
\label{c:19}
Suppose that $L\in\BL(\HH,\GG)$ satisfies $0<\|L\|\leq 1$, let 
$\beta\in\RPP$, let $\gamma\in\RPP$, and let
$g\colon\GG\to\RR$ be convex and $\beta$-Lipschitzian. Then
$L\proxcc{\gamma}g$ is $(\beta\|L\|)$-Lipschitzian.
\end{corollary}
\begin{proof}
We recall that a lower semicontinuous convex function 
$f\colon\HH\to\RR$ is $\beta$-Lipschitzian if
and only if $\ran\partial f=\dom\partial f^*\subset B(0;\beta)$
\cite[Corollary~17.19]{Livre1}. Since $g\in\Gamma_0(\GG)$,
Lemma~\ref{l:8}\ref{l:8i-} yields $g^*\in\Gamma_0(\GG)$. 
We therefore invoke Proposition~\ref{p:9}\ref{p:9ii} to get
\begin{align}
\label{e:p19}
\ran\partial\brk1{L\proxcc{\gamma}g}
&\subset L^*\brk2{\ran\brk1{\partial g^*
+\gamma(\Id_{\GG}-L\circ L^*)}^{-1}}\nonumber\\
&=L^*\brk2{\dom\brk1{\partial g^*
+\gamma(\Id_{\GG}-L\circ L^*)}}\nonumber\\
&=L^*\brk1{\dom\partial g^*}\nonumber\\
&\subset L^*\brk1{B(0;\beta)}\nonumber\\
&\subset B(0;\beta\|L\|),
\end{align}
where $L\proxcc{\gamma}g\colon\HH\to\RX$ is a real-valued
lower semicontinuous convex function by
Propositions~\ref{p:4}\ref{p:4iii+b} and \ref{p:7}\ref{p:7i}.
\end{proof}

Let us now evaluate Moreau envelopes of proximal cocompositions.

\begin{proposition}
\label{p:10}
Suppose that $L\in\BL(\HH,\GG)$ satisfies $0<\|L\|\leq 1$, let
$g\colon\GG\to\RX$ be a proper function such that $\cam g\neq\emp$,
let $\gamma\in\RPP$, and let $\rho\in\RPP$. Then the following hold:
\begin{enumerate}
\item
\label{p:10i}
$\moyo{(L\proxcc{\gamma+\rho}g)}{\rho}
=L\proxcc{\gamma}(\moyo{g}{\rho})$.
\item
\label{p:10ii}
$\moyo{(L\proxcc{\gamma}g)}{\gamma}
=\moyo{(g^{**})}{\gamma}\circ L$.
\end{enumerate}
\end{proposition}
\begin{proof}
By Lemma~\ref{l:1}\ref{l:1iv} and Proposition~\ref{p:7}\ref{p:7i},
$L\proxc{1/\gamma}g^*\in\Gamma_0(\HH)$. Therefore,
Lemma~\ref{l:8}\ref{l:8iv} and Definition~\ref{d:1} yield
\begin{equation}
\label{e:p5}
\brk2{\brk1{L\proxc{1/\gamma}g^*}+\rho\qq_{\HH}}^*
=\moyo{\brk2{\brk1{L\proxc{1/\gamma}g^*}^*}}{\rho}
=\moyo{\brk1{L\proxcc{\gamma}g}}{\rho}.
\end{equation}

\ref{p:10i}: We combine Definition~\ref{d:1},
Lemma~\ref{l:5}\ref{l:5i}, Proposition~\ref{p:5}\ref{p:5i}, and
\eqref{e:p5} to arrive at
\begin{equation}
L\proxcc{\gamma}\brk1{\moyo{g}{\rho}}
=\brk2{L\proxc{1/\gamma}\brk1{\moyo{g}{\rho}}^*}^*
=\brk2{L\proxc{1/\gamma}\brk1{g^*+\rho\qq_{\GG}}}^*
=\brk2{\brk1{L\proxc{1/(\gamma+\rho)}g^*}+\rho\qq_{\HH}}^*
=\moyo{\brk1{L\proxcc{\gamma+\rho}g}}{\rho}.
\end{equation}

\ref{p:10ii}: Since $g^*\in\Gamma_0(\GG)$, items \ref{l:8i-}
and \ref{l:8iii} in Lemma~\ref{l:8} imply
that $\moyo{(g^{**})}{\gamma}\in\Gamma_0(\GG)$ and that
$\dom \moyo{(g^{**})}{\gamma}=\GG$. Hence,
$\moyo{(g^{**})}{\gamma}\circ L\in\Gamma_0(\HH)$ and it follows
from Lemma~\ref{l:8}\ref{l:8i-}, Definition~\ref{d:1}, and
\eqref{e:p5} that
\begin{equation}
\moyo{\brk1{g^{**}}}{\gamma}\circ L
=\brk2{\moyo{\brk1{g^{**}}}{\gamma}\circ L}^{**}
=\brk2{\brk1{L\proxc{1/\gamma}g^*}+\gamma\qq_{\GG}}^*
=\moyo{\brk1{L\proxcc{\gamma}g}}{\gamma},
\end{equation}
as announced.
\end{proof}

\begin{corollary}
\label{c:argmin}
Suppose that $L\in\BL(\HH,\GG)$ satisfies $0<\|L\|\leq 1$, let
$g\colon\GG\to\RX$ be a proper function such that $\cam g\neq\emp$,
and let $\gamma\in\RPP$. Then 
$\Argmin(L\proxcc{\gamma}g)
=\Argmin(\moyo{(g^{**})}{\gamma}\circ L)$.
\end{corollary}
\begin{proof}
Since the set of minimizers of a function in $\Gamma_0(\HH)$
coincides with that of its Moreau envelope 
\cite[Propositions~17.5]{Livre1}, the assertion follows
from Proposition~\ref{p:10}\ref{p:10ii}.
\end{proof}

\begin{corollary}
\label{c:11}
Let $\KK$ be a real Hilbert space, suppose that 
$L\in\BL(\HH,\GG)$ and $S\in\BL(\KK,\HH)$ satisfy $\|L\|\leq 1$,
$\|S\|\leq 1$, and $L\circ S\neq 0$, let $g\colon\GG\to\RX$ be a
proper function such that $\cam g\neq\emp$, and let
$\gamma\in\RPP$. Then the following hold:
\begin{enumerate}
\item
\label{c:11i}
$S\proxcc{\gamma}(L\proxcc{\gamma}g)=(L\circ S)\proxcc{\gamma}g$.
\item
\label{c:11ii}
$S\proxc{\gamma}(L\proxc{\gamma}g)=(L\circ S)\proxc{\gamma}g$.
\end{enumerate}
\end{corollary}
\begin{proof}
\ref{c:11i}:
Set $f=L\proxcc{\gamma}g$. Since
$\|L\circ S\|\leq\|L\|\,\|S\|\leq 1$, we deduce from 
Proposition~\ref{p:7}\ref{p:7i} that $f\in\Gamma_0(\HH)$,
$S\proxcc{\gamma}f\in\Gamma_0(\KK)$, and 
$(L\circ S)\proxcc{\gamma}g\in\Gamma_0(\KK)$. By
Lemma~\ref{l:8}\ref{l:8i-}, $f^{**}=f$. Hence, 
Proposition~\ref{p:10}\ref{p:10ii} yields 
\begin{align}
\label{e:p317}
\moyo{\brk1{S\proxcc{\gamma}f}}{\gamma}
=\moyo{\brk1{f^{**}}}{\gamma}\circ S 
=\moyo{f}{\gamma}\circ S
=\brk2{\moyo{\brk1{g^{**}}}{\gamma}\circ L}\circ S
=\moyo{\brk2{\brk1{L\circ S}\proxcc{\gamma}g}}{\gamma}.
\end{align}
Therefore, the assertion follows from Lemma~\ref{l:6}.

\ref{c:11ii}:
By Proposition~\ref{p:7}\ref{p:7i},
$L\proxc{\gamma}g\in\Gamma_0(\HH)$,
$S\proxc{\gamma}(L\proxc{\gamma}g)\in\Gamma_0(\KK)$, and
$(L\circ S)\proxc{\gamma}g\in\Gamma_0(\KK)$. Therefore,
using Propositions~\ref{p:7}\ref{p:7iii} and
\ref{p:1}\ref{p:1iii}, together with \ref{c:11i}, we get
\begin{equation}
\label{e:p317b}
S\proxc{\gamma}\brk1{L\proxc{\gamma}g}
=\brk2{S\proxcc{1/\gamma}\brk1{L\proxc{\gamma}g}^*}^*
=\brk2{S\proxcc{1/\gamma}\brk1{L\proxcc{1/\gamma}g^*}}^*
=\brk2{\brk1{L\circ S}\proxcc{1/\gamma}g^*}^*
=\brk1{L\circ S}\proxc{\gamma}g,
\end{equation}
which completes the proof.
\end{proof}

\begin{proposition}
\label{p:13}
Suppose that $L\in\BL(\HH,\GG)$ satisfies $0<\|L\|\leq 1$, let
$g\colon\GG\to\RX$ be a proper function such that $\cam g\neq\emp$,
and let $\gamma\in\RPP$. Then
$\rec(L\proxcc{\gamma}g)=(\rec(g^{**}))\circ L$.
\end{proposition}
\begin{proof}
By Lemmas~\ref{l:1}\ref{l:1iv} and \ref{l:8}\ref{l:8i-},
$g^*\in\Gamma_0(\GG)$ and $g^{**}\in\Gamma_0(\GG)$. Therefore,
Lemma~\ref{l:8}\ref{l:8vi}, Propositions~\ref{p:7}\ref{p:7ii} and
\ref{p:4}\ref{p:4iii}, and Lemma~\ref{l:1}\ref{l:1iii}
imply that
\begin{equation}
\rec\brk1{L\proxcc{\gamma}g}
=\sigma_{\dom\brk1{L\proxcc{\gamma}g}^*}
=\sigma_{\dom\brk1{L\proxc{1/\gamma}g^*}}
=\sigma_{L^*(\dom g^{***})}=\sigma_{\dom g^*}\circ L
=\brk1{\rec(g^{**})}\circ L,
\end{equation}
as claimed.
\end{proof}

\begin{proposition}
\label{p:16}
Suppose that $L\in\BL(\HH,\GG)$ satisfies $0<\|L\|\leq 1$, let
$g\in\Gamma_0(\GG)$, let
\begin{equation}
\widetilde{g}\colon\GG\oplus\RR\to\RX\colon(y,\eta)\mapsto
\begin{cases}
\eta g(y/\eta),&\text{if}\;\;\eta>0;\\
\brk1{\rec g}(y),&\text{if}\;\;\eta=0;\\
\pinf,&\text{otherwise}
\end{cases}
\end{equation}
be its perspective, let $\gamma\in\RPP$, and set
$\widetilde{L}\colon\HH\oplus\RR\to\GG\oplus\RR\colon(x,\xi)
\mapsto(Lx,\xi)$. Then 
\begin{equation}
\label{e:p16}
\widetilde{L\proxcc{\gamma}g}\colon\HH\oplus\RR\to\RX\colon(x,\xi)
\mapsto
\begin{cases}
\brk2{\widetilde{L}\:\proxcc{\xi\gamma}\:\widetilde{g}\,}
(x,\xi),&\text{if}\quad\xi>0;\\
\brk1{\rec g}(Lx),&\text{if}\quad\xi=0;\\
\pinf,&\text{otherwise}.
\end{cases}
\end{equation}
\end{proposition}
\begin{proof}
Let $(x,\xi)\in\HH\oplus\RR$, set
$\Phi=\qq_{\GG}-\qq_{\HH}\circ L^*$, and set
$\Psi=\qq_{\GG\oplus\RR}-\qq_{\HH\oplus\RR}\circ\widetilde{L}\:^*$.
We consider two cases.
\begin{itemize}
\item
\label{p:16i}
$\xi=0$: It follows from Proposition~\ref{p:13} and
Lemma~\ref{l:8}\ref{l:8i--}--\ref{l:8i-}
that $(\widetilde{L\proxcc{\gamma}g})(x,0)
=(\rec(L\proxcc{\gamma}g))(x)=(\rec g)(Lx)$.
\item
\label{p:16ii}
$\xi>0$: Set 
$C=\menge{(y^*,\eta)\in\GG\oplus\RR}{\eta+g^*(y^*)\leq 0}$. Then
\cite[Items (ii) and (iv) in Proposition~2.3]{Svva18} assert that 
$\widetilde{g}\in\Gamma_0(\GG\oplus\RR)$ and
$(\widetilde{g}\,)^*=\iota_C$. Therefore, by
Lemma~\ref{l:2}\ref{l:2ii},
\begin{align}
\label{e:p16a}
(\forall y^*\in\GG)\quad
\sup_{\eta\in\RR}\brk1{\eta\xi-(\widetilde{g}\,)^*(y^*,\eta)}
&=\sup_{\eta\in\RR}\brk1{\eta\xi-\iota_C(y^*,\eta)}\nonumber\\
&=\sup_{\eta\in\left]\minf,-g^*(y^*)\right]}
\eta\xi\nonumber\\
&=-\xi g^*(y^*)\nonumber\\
&=-\brk1{\xi g(\cdot/\xi)}^*(y^*).
\end{align}
On the other hand, for every $\eta\in\RR$,
$\Psi(\cdot,\eta)=\Phi$ and, since $0<\|L\|\leq 1$, we have
$0<\|\widetilde{L}\|\leq 1$. Hence, appealing to
Proposition~\ref{p:4}\ref{p:4ii}, \eqref{e:p16a}, and
Proposition~\ref{p:1}\ref{p:1viii}--\ref{p:1ix},
\begin{align}
\label{e:p16b}
\brk2{\widetilde{L}\:\proxcc{\xi\gamma}\:\widetilde{g}\,}(x,\xi)
&=\brk1{(\widetilde{g}\,)^*+\xi\gamma\Psi}^*
\brk1{\widetilde{L}(x,\xi)}\nonumber\\
&=\brk1{(\widetilde{g}\,)^*+\xi\gamma\Psi}^*(Lx,\xi)\nonumber\\
&=\sup_{(y^*,\eta)\in\GG\oplus\RR}
\brk1{\scal{(Lx,\xi)}{(y^*,\eta)}_{\GG\oplus\RR}
-(\widetilde{g}\,)^*(y^*,\eta)-\xi\gamma\Psi(y^*,\eta)}\nonumber\\
&=\sup_{(y^*,\eta)\in\GG\oplus\RR}\brk1{\eta\xi+
\scal{Lx}{y^*}_{\GG}
-(\widetilde{g}\,)^*(y^*,\eta)-\xi\gamma\Phi(y^*)}\nonumber\\
&=\sup_{y^*\in\GG}\brk2{\scal{Lx}{y^*}_{\GG}-\xi\gamma\Phi(y^*)
+\sup_{\eta\in\RR}\brk1{\eta\xi-(\widetilde{g}\,)^*(y^*,\eta)}}
\nonumber\\
&=\sup_{y^*\in\GG}\brk1{\scal{Lx}{y^*}_{\GG}-\xi\gamma\Phi(y^*)
-\brk1{\xi g(\cdot/\xi)}^*(y^*)}\nonumber\\
&=\brk2{\brk1{\xi g(\cdot/\xi)}^*+\xi\gamma\Phi}^*(Lx)\nonumber\\
&=\brk2{L\proxcc{\xi\gamma}\brk1{\xi g(\cdot/\xi)}}(x)\nonumber\\
&=\xi\brk1{L\proxcc{\gamma}g}(x/\xi)\nonumber\\
&=\brk2{\,\widetilde{L\proxcc{\gamma}g}\,}(x,\xi).
\end{align}
\end{itemize}
We have thus proved \eqref{e:p16}.
\end{proof}

\subsection{Comparison with standard compositions and infimal
postcompositions}

As mentioned in Section~\ref{sec:1}, our discussion involves
several ways to compose a function defined on $\GG$ with a linear
operator from $\HH$ to $\GG$ in order to obtain a function defined
on $\HH$: the standard composition \eqref{e:1}, the infimal
postcomposition \eqref{e:2}, and the proximal composition and
cocomposition of Definition~\ref{d:1}. We saw in
Proposition~\ref{p:17} that a numerical advantage of the proximal
compositions is that their proximity operators are easily
decomposable in terms of that of the underlying function. Our
purpose here is to compare these various compositions.

\begin{figure}[htbp]
\centering
\includegraphics[width=11.1cm]{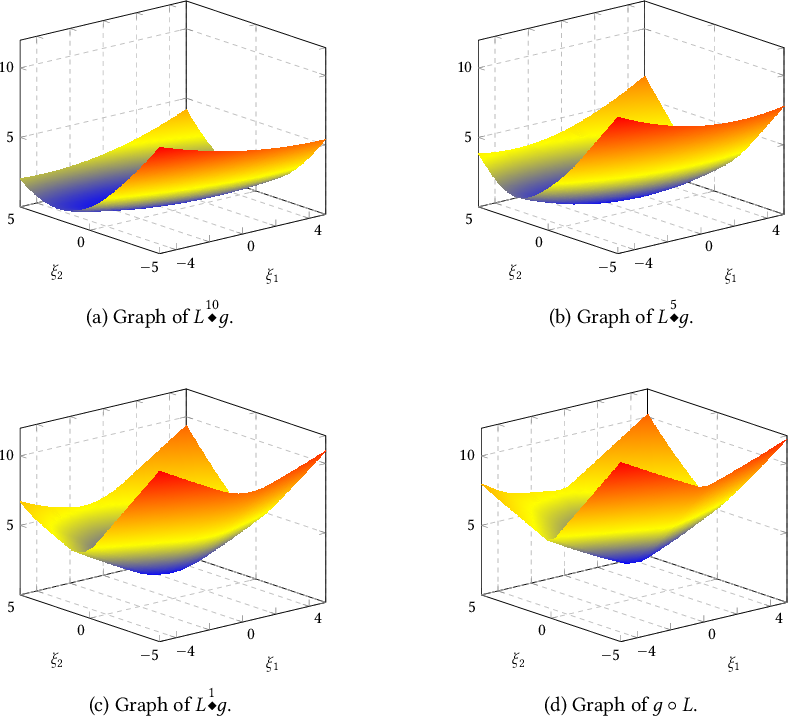}
\caption{Graphs of the proximal cocomposition and of the standard
composition in Example~\ref{ex:1}.}
\label{fig:1}
\end{figure}

\begin{figure}[htbp]
\centering
\includegraphics[width=11.1cm]{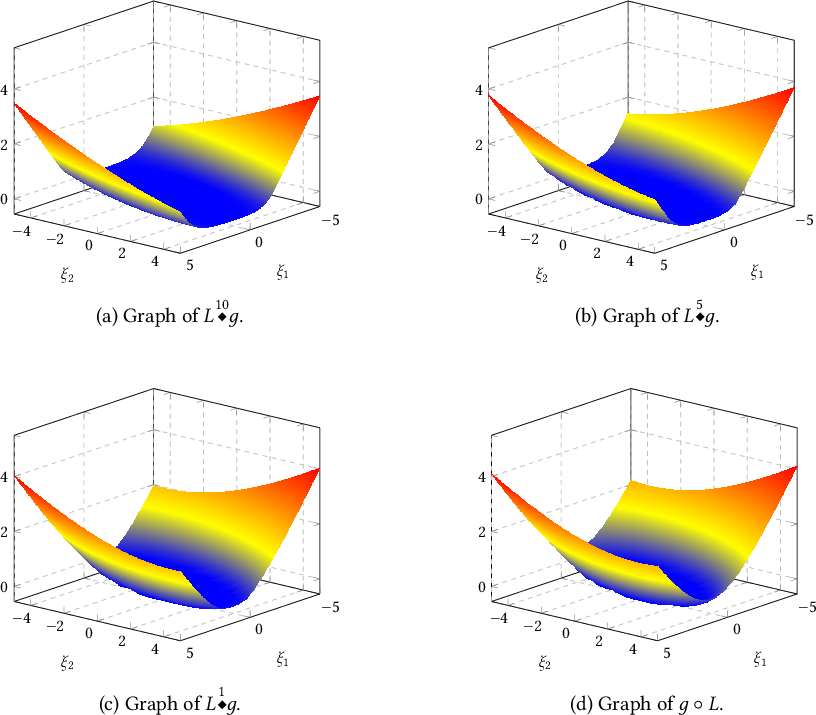}
\caption{Graphs of the proximal cocomposition and of the standard
composition in Example~\ref{ex:2}.}
\label{fig:2}
\end{figure}

\begin{example}
\label{ex:1}
Let 
\begin{equation}
\begin{cases}
L\colon\RR^2\to\RR^5\colon
(\xi_1,\xi_2)\mapsto\brk1{0.5\xi_2,-0.5\xi_1,-0.5\xi_2,
0.3\xi_1+0.4\xi_2,0.1\xi_1-0.3\xi_2}\\
g\colon\RR^5\to\RR\colon(\eta_1,\eta_2,\eta_3,\eta_4,\eta_5)\mapsto
\|(\eta_1,\eta_2,\eta_3)\|_1+\|(\eta_4-1,\eta_5+2)\|.
\end{cases}
\end{equation}
Figure~\ref{fig:1} shows the graphs of both the standard
composition and proximal cocomposition for various values of
$\gamma$. 
\end{example}

\begin{example}
\label{ex:2}
Let $C=B(0;2)$ and 
\begin{equation}
\begin{cases}
L\colon\RR^2\to\RR^3\colon(\xi_1,\xi_2)\mapsto
\brk1{0.7\xi_1+0.1\xi_2,-0.3\xi_1+0.4\xi_2,0.5\xi_1-0.3\xi_2}\\
g\colon\RR^3\to\RR\colon(\eta_1,\eta_2,\eta_3)\mapsto
d_C(\eta_1,\eta_2,\eta_3).
\end{cases}
\end{equation}
Figure~\ref{fig:2} shows the graphs of both the standard
composition and proximal cocomposition for various values of
$\gamma$. 
\end{example}

As we now show, the pointwise orderings suggested by 
Figures~\ref{fig:1} and \ref{fig:2} are generally true.

\begin{proposition}
\label{p:20}
Suppose that $L\in\BL(\HH,\GG)$ satisfies $0<\|L\|\leq 1$, let
$g\colon\GG\to\RX$ be a proper function such that $\cam g\neq\emp$,
and let $\gamma\in\RPP$. Then the following hold:
\begin{enumerate}
\item
\label{p:20i}
$L^*\pushfwd g^{**}\leq L\proxc{\gamma}g$.
\item
\label{p:20ii}
$\moyo{(g^{**})}{\gamma}\circ L\leq
L\proxcc{\gamma}g\leq g^{**}\circ L$.
\item
\label{p:20iii}
$L\proxcc{\gamma}g\leq L\proxc{\gamma}g$.
\item
\label{p:20iv}
Suppose that $L$ is an isometry. Then 
$L\proxc{\gamma}g=L\proxcc{\gamma}g$.
\item
\label{p:20v}
Suppose that $L$ is a coisometry. Then 
$L\proxc{\gamma}g=L^*\epushfwd g^{**}$ and~
$L\proxcc{\gamma}g=g^{**}\circ L$.
\item
\label{p:20vi}
Suppose that $L$ is invertible with $L^{-1}=L^*$. Then
$L\proxc{\gamma}g=L^*\epushfwd g^{**}=g^{**}\circ L
=L\proxcc{\gamma}g$.
\end{enumerate}
\end{proposition}
\begin{proof}
Set $\Phi=\qq_{\GG}-\qq_{\HH}\circ L^*$ and observe that
$0\leq\Phi\leq\qq_{\GG}$.

\ref{p:20i}: Let $x\in\HH$. By Proposition~\ref{p:4}\ref{p:4i},
\begin{equation}
\brk1{L\proxc{\gamma}g}(x)
=\min_{\substack{y\in\GG\\L^*y=x}}\brk3{g^{**}(y)
+\dfrac{1}{\gamma}\Phi(y)}
\geq\inf_{\substack{y\in\GG\\L^*y=x}}g^{**}(y)
=\brk1{L^*\pushfwd g^{**}}(x).
\end{equation}

\ref{p:20ii}: The leftmost inequality is established in
Proposition~\ref{p:4}\ref{p:4iv}. Let us prove rightmost
inequality. By Lemma~\ref{l:1}\ref{l:1ii} and \ref{p:20i},
$(L\proxc{1/\gamma}g^*)^*\leq(L^*\pushfwd g^{***})^*$. It
therefore follows from Definition~\ref{d:1} and
Lemmas~\ref{l:1}\ref{l:1iii} and \ref{l:5}\ref{l:5ii} that
\begin{equation}
L\proxcc{\gamma}g=\brk1{L\proxc{1/\gamma}g^*}^*
\leq(L^*\pushfwd g^*)^*=g^{**}\circ L.
\end{equation}

\ref{p:20iii}: Set $f=\moyo{(g^{**})}{1}\circ L$. Since
$\|L\|\leq 1$, $\qq_{\GG}\circ L\leq\qq_{\HH}$, and we deduce from
Lemma~\ref{l:1}\ref{l:1ii} that $(\qq_{\HH}-f)^*\leq
(\qq_{\GG}\circ L-f)^*$. However, it results from
Lemma~\ref{l:8}\ref{l:8ii+} that
$\qq_{\GG}\circ L-f=(\qq_{\GG}-\moyo{(g^{**})}{1})\circ L
=\moyo{(g^*)}{1}\circ L$. Altogether, it follows from 
Definition~\ref{d:1} and \cite[Proposition~13.29]{Livre1} that
\begin{equation}
\label{e:10}
L\proxcc{1}g=\brk1{f^*-\qq_{\HH}}^*=\brk1{\qq_{\HH}-f}^*-\qq_{\HH}
\leq\brk1{\moyo{(g^*)}{1}\circ L}^*-\qq_{\HH}
=L\proxc{1}{g}.
\end{equation}
Hence, by Proposition~\ref{p:1}\ref{p:1viii}, \eqref{e:10}, and
Proposition~\ref{p:1}\ref{p:1vi}, we get
\begin{equation}
L\proxcc{\gamma}g=\dfrac{1}{\gamma}\brk1{L\proxcc{1}(\gamma g)}
\leq\dfrac{1}{\gamma}\brk1{L\proxc{1}(\gamma g)}
=L\proxc{\gamma}g.
\end{equation}

\ref{p:20iv}: Here $\qq_{\HH}=\qq_{\GG}\circ L$ and therefore the
inequalities in the proof of \ref{p:20iii} can be replaced with
equalities.

\ref{p:20v}: Here $\qq_{\GG}=\qq_{\HH}\circ L^*$ and thus $\Phi=0$.
Therefore, the result follows from
Proposition~\ref{p:4}\ref{p:4i}--\ref{p:4ii}.

\ref{p:20vi}: A consequence of \ref{p:20iv} and \ref{p:20v}.
\end{proof}

\begin{remark}
Suppose that $L\in\BL(\HH,\GG)$ is an isometry, let
$g\colon\GG\to\RX$ be a proper function such that $\cam g\neq\emp$,
and let $\gamma\in\RPP$. Then we recover from
\cite[Proposition~13.24(v)]{Livre1} as well as items
\ref{p:20i}, \ref{p:20iv}, and \ref{p:20ii} in  
Proposition~\ref{p:20} the inequalities
\begin{equation}
(g^*\circ L)^*\leq L^*\pushfwd g^{**}\leq L\proxc{\gamma}g
=L\proxcc{\gamma}g\leq g^{**}\circ L,
\end{equation}
which appear in \cite[Proposition~5.4]{Svva23} in the special
case in which $\gamma=1$.
\end{remark}

\begin{remark}
Suppose that $L\in\BL(\HH,\GG)$ is a coisometry, let
$g\colon\GG\to\RX$ be a proper function such that $\cam g\neq\emp$,
and let $\gamma\in\RPP$. Then Propositions~\ref{p:20}\ref{p:20v}
and \ref{p:10}\ref{p:10ii} imply that
\begin{equation}
\label{e:337}
\moyo{\brk1{(g^{**})\circ L}}{\gamma}
=\moyo{\brk1{L\proxcc{\gamma}g}}{\gamma}
=\moyo{(g^{**})}{\gamma}\circ L.
\end{equation}
In particular, when $g\in\Gamma_0(\GG)$, we recover the fact that
$\moyo{(g\circ L)}{\gamma}=\moyo{g}{\gamma}\circ L$ 
(see \cite[Lemma~3]{Yama23}).
\end{remark}

\begin{proposition} 
\label{p:25} 
Suppose that $L\in\BL(\HH,\GG)$ satisfies $0<\|L\|\leq 1$, let
$g\colon\GG\to\RX$ be a proper function such that $\cam g\neq\emp$,
let $\gamma\in\RPP$, let $x\in\HH$, and set 
$\Phi=\qq_{\GG}-\qq_{\HH}\circ L^*$. Then the following hold:
\begin{enumerate}
\item
\label{p:25i}
Suppose that $y^*\in\partial g(Lx)$. Then
$0\leq g(Lx)-(L\proxcc{\gamma}g)(x)
\leq\gamma\Phi(y^*)$.
\item
\label{p:25ii}
Suppose that $0\in(\Id_{\GG}-L\circ L^*)(\partial g(Lx))$. Then
$(L\proxcc{\gamma}g)(x)=g(Lx)$.
\end{enumerate}
\end{proposition}
\begin{proof}
\ref{p:25i}:
By \cite[Proposition~16.10]{Livre1},
$g(Lx)+g^*(y^*)=\scal{Lx}{y^*}_{\GG}$. Further,
\cite[Proposition~16.5]{Livre1} yields $g^{**}(Lx)=g(Lx)\in\RR$.
Therefore, we deduce from Propositions~\ref{p:20}\ref{p:20ii} and
\ref{p:4}\ref{p:4ii} that $(L\proxcc{\gamma}g)(x)\in\RR$ and that
\begin{align}
0&\leq g(Lx)-\brk1{L\proxcc{\gamma}g}(x)\nonumber\\
&=g(Lx)-\brk1{g^*+\gamma\Phi}^*(Lx)\nonumber\\
&=g(Lx)-\sup_{y\in\GG}\brk1{\scal{Lx}{y}_{\GG}-g^*(y)
-\gamma\Phi(y)}\nonumber\\
&\leq g(Lx)-\brk1{\scal{Lx}{y^*}_{\GG}-g^*(y^*)
-\gamma\Phi(y^*)}\nonumber\\
&=\gamma\Phi(y^*).
\end{align}

\ref{p:25ii}:
There exists $y^*\in\partial g(Lx)$ such that $L(L^*y^*)=y^*$.
Therefore, $\Phi(y^*)=0$ and the conclusion follows from
\ref{p:25i}.
\end{proof} 

\begin{proposition}
\label{p:30}
Suppose that $L\in\BL(\HH,\GG)$ satisfies $0<\|L\|\leq 1$, let
$\beta\in\RPP$, let $\gamma\in\RPP$, and let $g\colon\GG\to\RR$ be
convex and $\beta$-Lipschitzian. Then the following hold:
\begin{enumerate}
\item
\label{p:30i}
$0\leq g\circ L-L\proxcc{\gamma}g\leq\gamma\beta^2/2$.
\item
\label{p:30ii}
$L^*\pushfwd g^*\leq L\proxc{1/\gamma}g^*\leq
(L^*\pushfwd g^*)+\gamma\beta^2/2$.
\item
\label{p:30iii}
Let $x\in\HH$. Then 
$\|\prox_{\gamma g\circ L}x-
\prox_{\gamma\brk1{L\proxcc{\gamma}g}}x\|\leq\gamma\beta$.
\end{enumerate}
\end{proposition}
\begin{proof}
We recall that a lower semicontinuous convex function 
$f\colon\HH\to\RR$ is $\beta$-Lipschitzian if
and only if $\ran\partial f=\dom\partial f^*\subset B(0;\beta)$
\cite[Corollary~17.19]{Livre1}. Moreover, since $\dom g=\GG$, we
have $\dom\partial g=\GG$ \cite[Proposition~16.27]{Livre1}.

\ref{p:30i}: Let $x\in\HH$ and set 
$\Phi=\qq_{\GG}-\qq_{\HH}\circ L^*$. Since $\dom\partial g=\GG$,
there exists 
$y^*\in\partial g(Lx)\subset\ran\partial g\subset B(0;\beta)$.
Thus, $\Phi(y^*)\leq\qq_{\GG}(y^*)\leq\beta^2/2$ and the result
follows from Proposition~\ref{p:25}\ref{p:25i}.

\ref{p:30ii}: The leftmost inequality follows from
Proposition~\ref{p:20}\ref{p:20i} and Lemma~\ref{l:1}\ref{l:1iii}.
On the other hand, Proposition~\ref{p:7}\ref{p:7i} implies that
$L\proxc{1/\gamma}g^*\in\Gamma_0(\HH)$. Additionally, in view of
Lemma~\ref{l:1}\ref{l:1ii} and \ref{p:30i},
$(L\proxcc{\gamma}g)^*\leq(g\circ L-\gamma\beta^2/2)^*$. Finally,
we deduce from Proposition~\ref{p:7}\ref{p:7ii} and 
\cite[Proposition~13.24(v)]{Livre1} that
\begin{equation}
L\proxc{1/\gamma}g^*=\brk1{L\proxcc{\gamma}g}^*
\leq\brk3{g\circ L-\dfrac{\gamma\beta^2}{2}}^*
=\brk1{g\circ L}^*+\dfrac{\gamma\beta^2}{2}
\leq\brk1{L^*\pushfwd g^*}+\dfrac{\gamma\beta^2}{2}.
\end{equation}

\ref{p:30iii}: Set $\Psi=\Id_{\GG}-L\circ L^*$,
$p_\gamma=\prox_{L\proxcc{}(\gamma g)}x$, and 
$p=\prox_{\gamma g\circ L}x$. We note that, since $\|L\|\leq 1$,
\begin{equation}
\label{e:23}
\Psi\;\text{is monotone and}\;\|\Psi\|\leq 1. 
\end{equation}
Next, we deduce from 
\cite[Proposition~16.44]{Livre1}
and Proposition~\ref{p:9}\ref{p:9ii} that there exists
$y_\gamma\in(\partial(\gamma g)^*+\Psi)^{-1}(Lp_\gamma)$ such that
$L^*y_\gamma=x-p_\gamma$. Thus, 
$y_\gamma\in\gamma\partial g(Lp_\gamma-\Psi y_\gamma)$. On
the other hand, by \cite[Proposition~16.44 and
Corollary~16.53(i)]{Livre1}, there exists 
$y\in\gamma\partial g(Lp)$ such that $L^*y=x-p$. Therefore, 
the monotonicity of $\gamma\partial g$ 
\cite[Theorem~20.25]{Livre1} entails that
\begin{equation}
\label{e:24}
\scal{(Lp_\gamma-\Psi y_\gamma)-Lp}{y_\gamma-y}_{\GG}\geq 0
\end{equation}
However, by \eqref{e:23}, the Cauchy--Schwarz
inequality, and the fact that
$\{y_\gamma,y\}\subset\gamma\ran\partial g\subset B(0;\gamma\beta)$
we derive that
\begin{align}
\label{e:p30iii}
\scal{(Lp_\gamma-\Psi y_\gamma)-Lp}{y_\gamma-y}_{\GG}\geq 0
&\Leftrightarrow\scal{p_\gamma-p}{L^*(y_\gamma-y)}_{\HH}
-\scal{\Psi y_\gamma}{y_\gamma-y}_{\GG}\geq 0\nonumber\\
&\Leftrightarrow\|p-p_\gamma\|_{\HH}^2\leq
\scal{\Psi y_\gamma}{y}_{\GG}-\scal{\Psi y_\gamma}{y_\gamma}_{\GG}
\nonumber\\
&\Rightarrow\|p-p_\gamma\|_{\HH}^2\leq
\scal{\Psi y_\gamma}{y}_{\GG}
\nonumber\\
&\Rightarrow\|p-p_\gamma\|_{\HH}^2\leq
\|\Psi\|\,\|y_\gamma\|_{\GG}\,\|y\|_{\GG}\nonumber\\
&\Rightarrow\|p-p_\gamma\|_{\HH}^2\leq(\gamma\beta)^2\nonumber\\
&\Leftrightarrow\|p-p_\gamma\|_{\HH}\leq\gamma\beta.
\end{align}
Since Proposition~\ref{p:1}\ref{p:1viii} asserts that
$\prox_{\gamma\brk1{L\proxcc{\gamma}g}}x=p_\gamma$, the proof
is complete. 
\end{proof}

\begin{example}
\label{ex:comp}
Let $L\in\BL(\HH,\GG)$, let $g\in\Gamma_0(\GG)$, let
$\gamma\in\RPP$, and let $\rho\in\RPP$. Suppose that $L\circ
L^*=\rho\Id_{\GG}$. Then the following hold:
\begin{enumerate}
\item
\label{ex:compi}
Set $h=g(\sqrt{\rho}\cdot)$ and $S=L/\sqrt{\rho}$. Then
$g\circ L=S\proxcc{\gamma}h$.
\item
\label{ex:compii}
$\prox_{\gamma g\circ L}=\Id_{\HH}+\rho^{-1}L^*\circ
(\prox_{\gamma\rho g}-\Id_{\GG})\circ L$.
\end{enumerate}
\end{example}
\begin{proof}
\ref{ex:compi}: Since $L\circ L^*=\rho\Id_{\GG}$, $S$ is a
coisometry, and we deduce from Proposition~\ref{p:20}\ref{p:20v}
and Lemma~\ref{l:8}\ref{l:8i-} that
$S\proxcc{\gamma}h=h\circ S=g\circ L$.

\ref{ex:compii}:
This follows from \ref{ex:compi} and 
Proposition~\ref{p:17}\ref{p:17ii} (see also 
\cite[Proposition~24.14]{Livre1}). 
\end{proof}

\begin{example}
\label{ex:proj}
Let $V$ be a closed vector subspace of $\HH$ and 
$\gamma\in\RPP$. Then the following hold:
\begin{enumerate}
\item
\label{ex:proji}
$\proj_V\proxc{\gamma}\,\|\cdot\|=\iota_V+\|\cdot\|$.
\item
\label{ex:projii}
$\proj_V\proxcc{\gamma}\,\|\cdot\|=\|\cdot\|\circ\proj_V$.
\end{enumerate}
\end{example}
\begin{proof}
Set $\Phi=\qq_{\HH}-\qq_{\HH}\circ\proj_V$ and let $x\in\HH$.

\ref{ex:proji}:
It follows from Proposition~\ref{p:4}\ref{p:4i},
Lemma~\ref{l:8}\ref{l:8i-}, and the identity
$\Phi=\qq_{\HH}\circ\proj_{V^\bot}$ that
\begin{equation}
\brk2{\proj_V\proxc{\gamma}\,\|\cdot\|}(x)
=\inf_{\substack{y\in\HH\\\proj_Vy=x}}
\brk3{\|y\|+\dfrac{1}{2\gamma}\|x-y\|^2}
=\begin{cases}
\|x\|,&\text{if}\;\;x\in V\\
\pinf,&\text{if}\;\;x\notin V
\end{cases}
~=\iota_V(x)+\|x\|.
\end{equation}

\ref{ex:projii}:
We recall that $\partial\|\cdot\|(x)=\{x/\|x\|\}$ if 
$x\neq 0$ and that $\partial\|\cdot\|(0)=B(0;1)$
\cite[Example~16.32]{Livre1}. Hence, 
\begin{align}
\label{e:ex2}
\proj_{V^\bot}\brk1{\partial\|\cdot\|(\proj_Vx)}
&=\begin{cases}
\bigl\{\proj_{V^\bot}\brk1{\proj_Vx/\|\proj_Vx\|}\bigr\},
&\text{if}\;\;\proj_Vx\neq 0;\\
\proj_{V^\bot}\brk1{B(0;1)},&\text{if}\;\;\proj_Vx=0
\end{cases}\nonumber\\
&=\begin{cases}
\{0\},&\text{if}\;\;x\notin V^\bot;\\
\proj_{V^\bot}\brk1{B(0;1)},&\text{if}\;\;x\in V^\bot
\end{cases}\nonumber\\
&\ni 0.
\end{align}
However, $\Id-\proj_V\circ\proj_V^*=\proj_{V^\bot}$. Therefore,
in view of Proposition~\ref{p:25}\ref{p:25ii}, this confirms that 
$\proj_V\proxcc{\gamma}\,\|\cdot\|=\|\cdot\|\circ\proj_V$. 
\end{proof}

\begin{remark}
In contrast with Proposition~\ref{p:20}\ref{p:20v},
Example~\ref{ex:proj}\ref{ex:projii} shows an instance in which the
proximal cocomposition coincides with the standard composition for
a linear operator which is not a coisometry.
\end{remark}

\subsection{Asymptotic properties}
We investigate the asymptotic properties of the families
$(L\proxc{\gamma}g)_{\gamma\in\RPP}$ and
$(L\proxcc{\gamma}g)_{\gamma\in\RPP}$ as $\gamma$ varies. These
results provide further connections between the compositions
\eqref{e:1}, \eqref{e:2}, and the proximal compositions of
Definition~\ref{d:1}.

\begin{proposition}
\label{p:40}
Suppose that $L\in\BL(\HH,\GG)$ satisfies $0<\|L\|\leq 1$ and let
$g\colon\GG\to\RX$ be a proper function such that 
$\cam g\neq\emp$. Suppose that $x\in L^{-1}(\dom g^{**})$ and set,
for every $\gamma\in\RPP$,
$x_\gamma=\prox_{\gamma(L\proxcc{\gamma}g)}x$. Then 
\begin{equation}
\displaystyle\lim_{0<\gamma\to 0}
\brk1{L\proxcc{\gamma}g}(x_\gamma)=g^{**}(Lx).
\end{equation}
\end{proposition}
\begin{proof}
We first observe that, by virtue of
Proposition~\ref{p:7}\ref{p:7i}, $(x_\gamma)_{\gamma\in\RPP}$ is
well defined. Appealing to Proposition~\ref{p:10}\ref{p:10ii}, we
get
\begin{equation}
\label{e:p41}
\brk1{L\proxcc{\gamma}g}(x_\gamma)
+\dfrac{1}{\gamma}\qq_{\HH}(x-x_\gamma)
=\moyo{\brk1{L\proxcc{\gamma}g}}{\gamma}(x)
=\moyo{\brk1{g^{**}}}{\gamma}(Lx).
\end{equation}
On the other hand, by Proposition~\ref{p:17}\ref{p:17ii},
\begin{equation}
\label{e:p42}
\dfrac{1}{\gamma}\qq_{\HH}(x-x_\gamma)
=\dfrac{1}{\gamma}\qq_{\HH}
\brk2{L^*\brk1{Lx-\prox_{\gamma g^{**}}(Lx)}}
\leq\dfrac{1}{\gamma}\|L\|^2\qq_{\GG}
\brk1{Lx-\prox_{\gamma g^{**}}(Lx)}.
\end{equation} 
Therefore, since $Lx\in\dom g^{**}$,
\cite[Proposition~12.33(iii)]{Livre1} implies that
$\gamma^{-1}\qq_{\HH}(x-x_\gamma)\to 0$ as 
$\gamma\to 0$. Finally, by \eqref{e:p41} and
\cite[Proposition~12.33(ii)]{Livre1},
\begin{equation}
\lim_{0<\gamma\to 0}\brk1{L\proxcc{\gamma}g}(x_\gamma)
=\lim_{0<\gamma\to 0}\moyo{\brk1{g^{**}}}{\gamma}(Lx)
=\brk1{g^{**}}(Lx),
\end{equation}
as claimed.
\end{proof}

\begin{theorem}
\label{t:45}
Suppose that $L\in\BL(\HH,\GG)$ satisfies $0<\|L\|\leq 1$, let
$g\colon\GG\to\RX$ be a proper function such that 
$\cam g\neq\emp$, and let $x\in\HH$. Then the following hold:
\begin{enumerate}
\item
\label{t:45i}
The function
$\RPP\to\RX\colon\gamma\mapsto(L\proxc{\gamma}g)(x)$
is decreasing.
\item
\label{t:45ii}
The function
$\RPP\to\RX\colon\gamma\mapsto(L\proxcc{\gamma}g)(x)$
is decreasing.
\item
\label{t:45iii}
$\displaystyle\lim_{\gamma\to\pinf}(L\proxc{\gamma}g)(x)
=(L^*\pushfwd g^{**})(x)$.
\item
\label{t:45iv}
$\displaystyle\lim_{0<\gamma\to 0}(L\proxcc{\gamma}g)(x)
=g^{**}(Lx)$. 
\item
\label{t:45vi}
Suppose that $\|L\|< 1$. Then 
$\displaystyle\lim_{\gamma\to\pinf}(L\proxcc{\gamma}g)(x)
=\inf_{y\in\GG}g^{**}(y)$.
\item
\label{t:45vii}
Suppose that $\|L\|=1$ and that $V=\ran(\Id_{\GG}-L\circ L^*)$ is
closed. Then
$\displaystyle\lim_{\gamma\to\pinf}(L\proxcc{\gamma}g)(x)
=\inf_{y\in Lx-V}g^{**}(y)$. 
\end{enumerate}
\end{theorem}
\begin{proof}
Set $\Phi=\qq_{\GG}-\qq_{\HH}\circ L^*$.

\ref{t:45i}: Fix $\gamma_1\in\RPP$ and $\gamma_2\in\RPP$ such
that $\gamma_1\leq\gamma_2$. Then we deduce from
Proposition~\ref{p:4}\ref{p:4i} that
\begin{equation}
L\proxc{\gamma_2}g=L^*\epushfwd\brk1{g^{**}+\Phi/\gamma_2}
\leq L^*\epushfwd\brk1{g^{**}+\Phi/\gamma_1}
=L\proxc{\gamma_1}g.
\end{equation}

\ref{t:45ii}: Fix $\gamma_1\in\RPP$ and $\gamma_2\in\RPP$ such
that $\gamma_1\leq\gamma_2$. By \ref{t:45i}, 
$L\proxc{1/\gamma_1}g^*\leq L\proxc{1/\gamma_2}g^*$. Therefore,
appealing to Definition~\ref{d:1} and Lemma~\ref{l:1}\ref{l:1ii},
we get 
\begin{equation}
L\proxcc{\gamma_2}g=\brk1{L\proxc{1/\gamma_2}g^*}^*
\leq\brk1{L\proxc{1/\gamma_1}g^*}^*=L\proxcc{\gamma_1}g.
\end{equation}

\ref{t:45iii}: Since $\Phi\geq 0$, it follows from \ref{t:45i} and
Proposition~\ref{p:4}\ref{p:4i} that
\begin{align}
\lim_{\gamma\to\pinf}\brk1{L\proxc{\gamma}g}(x)
&=\inf_{\gamma\in\RPP}\brk3{L^*\epushfwd\brk2{g^{**}
+\dfrac{1}{\gamma}\Phi}}(x)\nonumber\\
&=\inf_{\gamma\in\RPP}\brk4{\inf_{\substack{y\in\GG\\L^*y=x}}
\brk2{g^{**}(y)+\dfrac{1}{\gamma}\Phi(y)}}\nonumber\\
&=\inf_{\substack{y\in\GG\\L^*y=x}}\brk4{\inf_{\gamma\in\RPP}
\brk2{g^{**}(y)+\dfrac{1}{\gamma}\Phi(y)}}\nonumber\\
&=\inf_{\substack{y\in\GG\\L^*y=x}}g^{**}(y)\nonumber\\
&=\brk1{L^*\pushfwd g^{**}}(x).
\end{align}

\ref{t:45iv}: 
By \cite[Proposition~12.33(ii)]{Livre1}, 
$\moyo{\brk1{g^{**}}}{\gamma}\to g^{**}$ as $0<\gamma\to 0$. The
claim therefore follows from Proposition~\ref{p:20}\ref{p:20ii}.

\ref{t:45vi}--\ref{t:45vii}: As in the proof of
Proposition~\ref{p:4}\ref{p:4iii+},
$(g^*+\gamma\Phi)^*=g^{**}\infconv(\Phi^*/\gamma)$. Thus, it
follows from Proposition~\ref{p:4}\ref{p:4ii} that
\begin{equation}
\label{e:t45v}
L\proxcc{\gamma}g=\brk2{g^{**}\infconv\brk1{\Phi^*/\gamma}}\circ L.
\end{equation}
Moreover, since $\Phi\leq\qq_{\GG}$, Lemma~\ref{l:1}\ref{l:1ii}
yields $\qq_{\GG}\leq\Phi^*$. Altogether, using \ref{t:45ii} and
\eqref{e:t45v}, we obtain
\begin{align}
\label{e:t45vi}
\lim_{\gamma\to\pinf}\brk1{L\proxcc{\gamma}g}(x)
&=\inf_{\gamma\in\RPP}\brk3{g^{**}\infconv
\dfrac{\Phi^*}{\gamma}}(Lx)\nonumber\\
&=\inf_{\gamma\in\RPP}\brk3{\inf_{y\in\GG}\brk2{g^{**}(y)
+\dfrac{1}{\gamma}\Phi^*(Lx-y)}}\nonumber\\
&=\inf_{y\in Lx-\dom\Phi^*}\brk3{\inf_{\gamma\in\RPP}
\brk2{g^{**}(y)+\dfrac{1}{\gamma}\Phi^*(Lx-y)}}\nonumber\\
&=\inf_{y\in Lx-\dom\Phi^*}g^{**}(y).
\end{align}
We set $A=\Id_{\GG}-L\circ L^*$ and observe that
$\Phi\colon y\mapsto\scal{y}{Ay}_{\GG}/2$. In case \ref{t:45vi},
since $\|L\|<1$, $A$ is invertible and Lemma~\ref{l:10} asserts
that $\dom\Phi^*=\ran A=\GG$ in \eqref{e:t45vi}. Finally, case
\ref{t:45vii} follows from Lemma~\ref{l:10} and \eqref{e:t45vi}. 
\end{proof}

\begin{corollary}
\label{c:46}
Suppose that $L\in\BL(\HH,\GG)$ is an isometry, let
$g\in\Gamma_0(\GG)$, and let $x\in\HH$. Then the following hold:
\begin{enumerate}
\item
\label{c:46ii}
$\displaystyle\lim_{\gamma\to\pinf}(L\proxc{\gamma}g)(x)
=(L^*\pushfwd g)(x)$.
\item
\label{c:46i}
$\displaystyle\lim_{0<\gamma\to 0}(L\proxc{\gamma}g)(x)=g(Lx)$.
\end{enumerate}
\end{corollary}
\begin{proof}
By Proposition~\ref{p:20}\ref{p:20iv},
$L\proxcc{\gamma}g=L\proxc{\gamma}g$, whereas
Lemma~\ref{l:8}\ref{l:8i-} yields $g^{**}=g$.

\ref{c:46ii}: A consequence of Theorem~\ref{t:45}\ref{t:45iii}.

\ref{c:46i}: A consequence of Theorem~\ref{t:45}\ref{t:45iv}.
\end{proof}

\begin{example}
Let $V\neq\{0\}$ be a closed vector subspace of $\GG$, let
$g\in\Gamma_0(\GG)$, and let $x\in\GG$. Then
\begin{equation}
\displaystyle\lim_{\gamma\to\pinf}
\brk1{\proj_V\proxcc{\gamma}\,g}(x)=\inf_{v\in V^\bot}g(x+v).
\end{equation}
\end{example}
\begin{proof}
Since $\|\proj_V\|=1$ and
$\ran(\Id_{\GG}-\proj_V\circ\proj_V^*)=V^\bot$, it follows from
Theorem~\ref{t:45}\ref{t:45vii} and Lemma~\ref{l:8}\ref{l:8i-} that
\begin{equation}
\lim_{\gamma\to\pinf}\brk1{\proj_V\proxcc{\gamma}\,g}(x)
=\inf_{y\in \proj_Vx-V^\bot}g(y)
=\inf_{y\in x+V^\bot}g(y)
=\inf_{v\in V^\bot}g(x+v),
\end{equation}
as announced.
\end{proof}

We now turn our attention to epi-convergence. As discussed in
\cite{Atto84}, this notion plays a central role in the
approximation of variational problems. It will allow us to connect
asymptotically the proximal composition to the infimal
postcomposition, and the proximal cocomposition to the standard
composition as $\gamma$ evolves. 

\begin{definition}[{\protect{\cite[Chapter~1]{Atto84},
\cite[Chapter~7]{Rock09}}}]
Suppose that $\HH$ is finite-dimensional, and let $(f_n)_{n\in\NN}$
and $f$ be functions from $\HH$ to $\RXX$. We say that
$(f_n)_{n\in\NN}$ \emph{epi-converges} to $f$, in symbols
$f_n\xrightarrow{e}f$, if the following hold for every $x\in\HH$:
\begin{enumerate}
\item
For every sequence $(x_n)_{n\in\NN}$ in $\HH$ such that 
$x_n\to x$, $f(x)\leq\varliminf f_n(x_n)$.
\item
There exists a sequence $(x_n)_{n\in\NN}$ in $\HH$ such that 
$x_n\to x$ and $\varlimsup f_n(x_n)\leq f(x)$.
\end{enumerate}
The \emph{epi-topology} is the topology induced by epi-convergence.
\end{definition}

\begin{lemma}
\label{l:47}
Suppose that $\HH$ and $\GG$ are finite-dimensional, let
$(L_n)_{n\in\NN}$ and $L$ be operators in $\BL(\HH,\GG)$, let 
$(g_n)_{n\in\NN}$ and $g$ be functions in $\Gamma_0(\GG)$, and
let $(\gamma_n)_{n\in\NN}$ and $\gamma$ be reals in $\RPP$.
Suppose that $L_n\to L$, $g_n\xrightarrow{e}g$, and
$\gamma_n\to\gamma$. Then the following hold:
\begin{enumerate}
\item
\label{l:47i}
$\gamma_n g_n\xrightarrow{e}\gamma g$.
\item
\label{l:47ii}
$g_n^*\xrightarrow{e}g^*$.
\item
\label{l:47iii}
Suppose that $h\colon\GG\to\RR$ is continuous. Then 
$g_n+\gamma_nh\xrightarrow{e}g+\gamma h$.
\item
\label{l:47iv}
Suppose that $0\in\inte(\dom g-\ran L)$. Then 
$g_n\circ L_n\xrightarrow{e}g\circ L$.
\end{enumerate}
\end{lemma}
\begin{proof}
\ref{l:47i}: \cite[Exercise~7.8(d)]{Rock09}.

\ref{l:47ii}: \cite[Theorem~11.34]{Rock09}.

\ref{l:47iii}: It follows from \ref{l:47i} and
\cite[Exercise~7.8(a)]{Rock09} that
$g_n/\gamma_n+h\xrightarrow{e}g/\gamma+h$.
Invoking \ref{l:47i} once more, we obtain
$g_n+\gamma_n h=\gamma_n(g_n/\gamma_n+h)\xrightarrow{e}
\gamma(g/\gamma+h)=g+\gamma h$.

\ref{l:47iv}: \cite[Exercise~7.47(a)]{Rock09}.
\end{proof}

\begin{theorem}
\label{t:50}
Suppose that $\HH$ and $\GG$ are finite-dimensional, let
$(L_n)_{n\in\NN}$ and $L$ be operators in $\BL(\HH,\GG)$, let 
$(g_n)_{n\in\NN}$ and $g$ be functions in $\Gamma_0(\GG)$, and
let $(\gamma_n)_{n\in\NN}$ and $\gamma$ be reals in $\RPP$. Then
the following hold:
\begin{enumerate}
\item
\label{t:50i}
Suppose that $L_n\to L$, $g_n\xrightarrow{e}g$, and
$\gamma_n\to\gamma$. Then the following are satisfied:
\begin{enumerate}
\item
\label{t:50ia}
$L_n\proxc{\gamma_n}g_n\xrightarrow{e}L\proxc{\gamma}g$.
\item
\label{t:50ib}
Suppose that $\|L\|\in\rzeroun$ and $(\forall n\in\NN)$ 
$\|L_n\|\in\rzeroun$.
Then $L_n\proxcc{\gamma_n}g_n\xrightarrow{e}L\proxcc{\gamma}g$.
\end{enumerate}
\item
\label{t:50ii}
Suppose that $0<\|L\|\leq 1$. Then the following are satisfied:
\begin{enumerate}
\item
\label{t:50iia}
Suppose that $\gamma_n\uparrow\pinf$. Then
$L\proxc{\gamma_n}g\xrightarrow{e}
(L^*\pushfwd g)^{{\mbox{\raisebox{-1mm}{\large$\breve{}$}}}}$.
\item
\label{t:50iib}
Suppose that $\gamma_n\downarrow 0$. Then
$L\proxcc{\gamma_n}g\xrightarrow{e}g\circ L$.
\end{enumerate}
\end{enumerate}
\end{theorem}
\begin{proof}
\ref{t:50ia}: It follows from Lemmas~\ref{l:8}\ref{l:8iv} and
\ref{l:47}\ref{l:47ii}--\ref{l:47iii} that
\begin{equation}
\label{e:epi1}
\moyo{\brk1{g_n^*}}{\frac{1}{\gamma_n}}
=\brk2{g_n+\dfrac{1}{\gamma_n}\qq_{\GG}}^*\xrightarrow{e}
\brk2{g+\dfrac{1}{\gamma}\qq_{\GG}}^*
=\moyo{\brk1{g^*}}{\frac{1}{\gamma}}.
\end{equation}
Since Lemmas~\ref{l:1}\ref{l:1iv} and \ref{l:8}\ref{l:8iii} yield
$\dom\moyo{(g^*)}{\frac{1}{\gamma}}=\GG$,
Lemma~\ref{l:47}\ref{l:47iv} and \eqref{e:epi1} imply that 
$\moyo{(g_n^*)}{\frac{1}{\gamma_n}}\circ L_n\xrightarrow{e}
\moyo{(g^*)}{\frac{1}{\gamma}}\circ L$.
Finally, appealing to Definition~\ref{d:1} and
Lemma~\ref{l:47}\ref{l:47ii}--\ref{l:47iii}, we conclude that
\begin{equation}
L_n\proxc{\gamma_n}g_n
=\brk2{\moyo{\brk1{g_n^*}}{\frac{1}{\gamma_n}}\circ L_n}^*
-\dfrac{1}{\gamma_n}\qq_{\HH}
\xrightarrow{e}
\brk2{\moyo{\brk1{g^*}}{\frac{1}{\gamma}}\circ L}^*
-\dfrac{1}{\gamma}\qq_{\HH}
=L\proxc{\gamma}g.
\end{equation}

\ref{t:50ib}: By Lemma~\ref{l:47}\ref{l:47ii},
$g_n^*\xrightarrow{e}g^*$. Therefore, upon combining \ref{t:50ia}
and Lemma~\ref{l:47}\ref{l:47ii}, we obtain
\begin{equation}
L_n\proxcc{\gamma_n}g_n=\brk1{L_n\proxc{1/\gamma_n}g_n^*}^*
\xrightarrow{e}\brk1{L\proxc{1/\gamma}g^*}^*=L\proxcc{\gamma}g.
\end{equation}

\ref{t:50iia}: 
Set $f=L^*\pushfwd g$ and $(\forall n\in\NN)$
$f_n=L\proxc{\gamma_n}g$. It follows from items \ref{t:45i} and
\ref{t:45iii} in Theorem~\ref{t:45}, as well as
Lemma~\ref{l:8}\ref{l:8i-}, that $(f_n)_{n\in\NN}$ is decreasing
and pointwise convergent to $f$ as $n\to\pinf$. Further, since $f$
is convex by \cite[Proposition~12.36(ii)]{Livre1}, we deduce from
\cite[Proposition~7.4(c)]{Rock09} and \cite[Corollary~9.10]{Livre1}
that
\begin{equation}
f_n\xrightarrow{e}\overline{\inf_{n\in\NN}f_n}
=\overline{f}
=\breve{f}.
\end{equation}

\ref{t:50iib}: 
Set $f=g\circ L$ and $(\forall n\in\NN)$
$f_n=L\proxcc{\gamma_n}g$. Since $(\gamma_n)_{n\in\NN}$ is
decreasing, $(f_n)_{n\in\NN}$ is increasing by
Theorem~\ref{t:45}\ref{t:45ii}. Further,
Theorem~\ref{t:45}\ref{t:45iv} and Lemma~\ref{l:8}\ref{l:8i-} 
imply that $(f_n)_{n\in\NN}$ converges pointwise to $f$ as
$n\to\pinf$. On the other hand, Proposition~\ref{p:7}\ref{p:7i}
implies that $(\forall n\in\NN)$ $\overline{f_n}=f_n$. Therefore,
by virtue of \cite[Proposition~7.4(d)]{Rock09},
\begin{equation}
f_n\xrightarrow{e}
\sup_{n\in\NN}\overline{f_n}
=\sup_{n\in\NN}f_n
=f,
\end{equation}
which concludes the proof.
\end{proof}

\begin{corollary}
\label{c:51}
Suppose that $\HH$ and $\GG$ are finite-dimensional, let 
$L\in\BL(\HH,\GG)$, let $g\in\Gamma_0(\GG)$, and let 
$(\gamma_n)_{n\in\NN}$ be a sequence in $\RPP$. Suppose that 
$L$ is an isometry and that $(\reli\dom g^*)\cap(\ran L)\neq\emp$.
Then the following hold: 
\begin{enumerate}
\item
\label{c:51i}
Suppose that $\gamma_n\uparrow\pinf$. Then
$L\proxc{\gamma_n}g\xrightarrow{e}L^*\pushfwd g$.
\item
\label{c:51ii}
Suppose that $\gamma_n\downarrow 0$. Then
$L\proxc{\gamma_n}g\xrightarrow{e}g\circ L$.
\item
\label{c:51iii}
For every $t\in[0,1]$, set $\gamma_t=\tan(\pi t/2)$. 
Then the operator 
\begin{equation}
T\colon[0,1]\to\Gamma_0(\HH)\colon t\to
\begin{cases}
g\circ L,&\text{if}\;\;t=0;\\
L\proxc{\gamma_t}g,&\text{if}\;\;0<t<1;\\
L^*\pushfwd g,&\text{if}\;\;t=1
\end{cases}
\end{equation}
is continuous with respect to the epi-topology.
\end{enumerate}
\end{corollary}
\begin{proof}
Proposition~\ref{p:20}\ref{p:20iv} yields 
$(\forall\gamma\in\RPP)$ $L\proxcc{\gamma}g=L\proxc{\gamma}g$.
Further, \cite[Proposition~6.19(x)]{Livre1} implies that 
$0\in\sri(\dom g^*-\ran L)$. Therefore, by virtue of
Lemmas~\ref{l:5}\ref{l:5iii} and \ref{l:8}\ref{l:8i-},
we get $L^*\pushfwd g\in\Gamma_0(\HH)$.

\ref{c:51i}: A consequence of Theorem~\ref{t:50}\ref{t:50iia}.

\ref{c:51ii}: See Theorem~\ref{t:50}\ref{t:50iib}.

\ref{c:51iii}: Theorem~\ref{t:50}\ref{t:50ia} guarantees the
epi-continuity of $T$ on $\zeroun$. Finally, 
\ref{c:51i} and \ref{c:51ii} imply that
$\displaystyle\lim_{0<t\to 0}T(t)=T(0)$ and
$\displaystyle\lim_{1>t\to1}T(t)=T(1)$, respectively.
\end{proof}

\begin{remark} 
Suppose that $\HH$ and $\GG$ are finite-dimensional and that
$L\in\BL(\HH,\GG)$ satisfies $0<\|L\|\leq 1$, let
$g\in\Gamma_0(\GG)$, and let $(\gamma_n)_{n\in\NN}$ be a sequence
in $\RPP$. Under a qualification condition (see
Lemma~\ref{l:5}\ref{l:5iii}), $L^*\pushfwd g\in\Gamma_0(\HH)$ and,
consequently, $L^*\pushfwd g
=(L^*\pushfwd g)^{{\mbox{\raisebox{-1mm}{\large$\breve{}$}}}}$. In
this case, Theorem~\ref{t:45}\ref{t:45iii} and
Theorem~\ref{t:50}\ref{t:50iia} show that the proximal composition
converges pointwise and epi-converges to the infimal
postcomposition as $\gamma_n\uparrow\pinf$. On the other hand,
Theorem~\ref{t:45}\ref{t:45iv} and Theorem~\ref{t:50}\ref{t:50iib}
show that the proximal cocomposition converges pointwise and
epi-converges to the standard composition. Further, in the
particular case in which $L\in\BL(\HH,\GG)$ is an isometry,
Corollary~\ref{c:51}\ref{c:51iii} asserts that $g\circ L$ and
$L^*\pushfwd g$ are homotopic via the proximal composition with
respect to the epi-topology. 
\end{remark}

\begin{proposition}
\label{p:55}
Suppose that $\HH$ and $\GG$ are finite-dimensional and that 
$L\in\BL(\HH,\GG)$ satisfies $0<\|L\|\leq 1$, let
$g\in\Gamma_0(\GG)$, and let $(\gamma_n)_{n\in\NN}$ be a sequence
in $\RPP$ such that $\gamma_n\downarrow 0$. Suppose that $\dom
g\cap\ran L\neq\emp$ and that $g\circ L$ is coercive. Then the
following hold:
\begin{enumerate}
\item
\label{p:55i}
$\inf_{x\in\HH}(L\proxcc{\gamma_n}g)(x)\to\min_{x\in\HH}g(Lx)$.
\item
\label{p:55ii}
There exists $N\subset\NN$ such that $\NN\smallsetminus N$ is
finite and $(\forall n\in N)$
$\Argmin(L\proxcc{\gamma_n}g)\neq\emp$. Further, 
\begin{equation}
\varlimsup\Argmin\brk1{L\proxcc{\gamma_n}g}\subset
\Argmin\brk1{g\circ L}.
\end{equation}
\end{enumerate}
\end{proposition}
\begin{proof}
Set $f=g\circ L$ and $(\forall n\in\NN)$ $f_n=L\proxcc{\gamma_n}g$.
Since $\dom g\cap\ran L\neq\emp$, $f\in\Gamma_0(\HH)$. Thus, by
\cite[Proposition~11.15(i)]{Livre1}, $f$ has a minimizer over
$\HH$. Further, by Proposition~\ref{p:7}\ref{p:7i}, for every
$n\in\NN$, $f_n\in\Gamma_0(\HH)$ and, by
Theorem~\ref{t:50}\ref{t:50iib}, $f_n\xrightarrow{e}f$. On the
other hand, \cite[Proposition~11.12]{Livre1} asserts that the lower
level sets $(\lev{\xi}f)_{\xi\in\RR}$ are bounded. Altogether, by
virtue of \cite[Exercise~7.32(c)]{Rock09}, for every $\xi\in\RR$,
there exists $N_\xi\in\NN$ such that  
$\bigcup_{n\geq N_\xi}\lev{\xi}f_n$ is bounded.

\ref{p:55i}--\ref{p:55ii}: A consequence of
\cite[Theorem~7.33]{Rock09}.
\end{proof}

\section{Integral proximal mixtures}
\label{sec:4}

\subsection{Definition and mathematical setting} 
\label{sec:41}
Integral proximal mixtures were introduced in \cite{Jota24} as a
tool to combine arbitrary families of convex functions and linear
operators in such a way that the proximity operator of the mixture
can be expressed explicitly in terms of the individual proximity
operators. They extend the proximal mixtures of \cite{Svva23},
which were designed for finite families. In this section, we use
the results of Section~\ref{sec:3} to study their variational
properties. This investigation is carried out in the same framework
as in \cite{Jota24}, which hinges on the following assumptions.
Henceforth, we adopt the customary convention that the integral
of an $\FF$-measurable function $\vartheta\colon\Omega\to\RXX$ is
the usual Lebesgue integral $\int_{\Omega}\vartheta d\mu$, except
when the Lebesgue integral $\int_{\Omega}\max\{\vartheta,0\}d\mu$
is $\pinf$, in which case $\int_{\Omega}\vartheta d\mu=\pinf$. 

\begin{assumption}
\label{a:1}
Let $(\Omega,\FF,\mu)$ be a complete $\sigma$-finite measure space,
let $(\GW)_{\omega\in\Omega}$ be a family of real Hilbert spaces,
and let $\prod_{\omega\in\Omega}\GW$ be the usual real vector
space of mappings $x$ defined on $\Omega$ such that
$(\forall\omega\in\Omega)$ $x(\omega)\in\GW$.
Let $((\GW)_{\omega\in\Omega},\mathfrak{G})$ be an
$\FF$-measurable vector field of real Hilbert spaces,
that is, $\mathfrak{G}$ is a vector subspace of
$\prod_{\omega\in\Omega}\GW$ which satisfies the following:
\begin{enumerate}[label={\normalfont[\Alph*]}]
\item
\label{a:1A}
For every $x\in\mathfrak{G}$, the function $\Omega\to\RR\colon
\omega\mapsto\norm{x(\omega)}_{\GW}$ is
$\FF$-measurable.
\item
\label{a:1B}
For every $x\in\prod_{\omega\in\Omega}\GW$,
\begin{equation}
\brk[s]!{\;(\forall y\in\mathfrak{G})\;\;
\Omega\to\RR\colon
\omega\mapsto\scal{x(\omega)}{y(\omega)}_{\GW}
\,\,\text{is $\FF$-measurable}\;}
\quad\Rightarrow\quad x\in\mathfrak{G}.
\end{equation}
\item
\label{a:1C}
There exists a sequence $(e_n)_{n\in\NN}$ in $\mathfrak{G}$ such
that $(\forall\omega\in\Omega)$
$\spc\{e_n(\omega)\}_{n\in\NN}=\GW$.
\end{enumerate}
Set
$\GH=\menge{x\in\mathfrak{G}}{\int_{\Omega}
\norm{x(\omega)}_{\GW}^2\mu(d\omega)<\pinf}$,
and let $\GG$ be the real Hilbert space of equivalence classes of
$\mae$ equal mappings in $\GH$ equipped with the scalar product
\begin{equation}
\label{e:y9d2}
\scal{\Cdot}{\Cdot}_{\GG}\colon\GG\times\GG\to\RR\colon
(x,y)\mapsto\int_{\Omega}
\scal{x(\omega)}{y(\omega)}_{\GW}\mu(d\omega),
\end{equation}
where we adopt the common practice of designating by $x$ both an
equivalence class in $\GG$ and a representative of it in $\GH$.
We write
\begin{equation}
\GG=\leftindex^{\mathfrak{G}}{\int}_{\Omega}^{\oplus}
\GW\mu(d\omega)
\end{equation}
and call $\GG$ the \emph{Hilbert direct integral of}
$((\GW)_{\omega\in\Omega},\mathfrak{G})$ \cite{Dixm69}.
\end{assumption}

\begin{assumption}
\label{a:3}
Assumption~\ref{a:1} and the following are in force:
\begin{enumerate}[label={\normalfont[\Alph*]}]
\item
\label{a:3a}
$\HS$ is a separable real Hilbert space.
\item
\label{a:3b}
For every $\omega\in\Omega$, $\LW\in\BL(\HS,\GW)$.
\item
\label{a:3c}
For every $\mathsf{x}\in\HS$, the mapping
$\mathfrak{e}_{\LS}\mathsf{x}\colon
\omega\mapsto\LW\mathsf{x}$ lies in $\mathfrak{G}$.
\item
\label{a:3d}
$0<\int_{\Omega}\norm{\LW}^2\mu(d\omega)\leq 1$.
\end{enumerate}
\end{assumption}

Given a complete $\sigma$-finite measure space $(\Omega,\FF,\mu)$,
a separable real Hilbert space $\HS$ with Borel $\sigma$-algebra
$\BE_{\HS}$, and $p\in\intv[r]{1}{\pinf}$, we set 
\begin{equation}
\mathscr{L}^p\brk1{\Omega,\FF,\mu;\HS}
=\Menge3{x\colon\Omega\to\HS}{x\,\,
\text{is $(\FF,\BE_{\HS})$-measurable and}\,\,
\int_{\Omega}\norm{x(\omega)}_{\HS}^p\,\mu(d\omega)<\pinf}.
\end{equation}
The Lebesgue integral (also known as the Bochner integral) of 
$x\in\mathscr{L}^1\brk1{\Omega,\FF,\mu;\HS}$ is denoted by
$\int_{\Omega}x(\omega)\mu(d\omega)$. The space of equivalence
classes of $\mae$ equal mappings in
$\mathscr{L}^p\brk{\Omega,\FF,\mu;\HS}$ is denoted by 
$L^p\brk{\Omega,\FF,\mu;\HS}$.

\begin{assumption}
\label{a:2}
Assumption~\ref{a:1} and the following are in force:
\begin{enumerate}[label={\normalfont[\Alph*]}]
\item
\label{a:2a}
For every $\omega\in\Omega$, $\gw\colon\GW\to\RX$ satisfies
$\cam\gw\neq\emp$.
\item
\label{a:2d}
For every $x^*\in\GH$, the mapping
$\omega\mapsto\prox_{\gw^*}x^*(\omega)$ lies in $\mathfrak{G}$.
\item
\label{a:2b}
There exists $r\in\GH$ such that the function
$\omega\mapsto\gw(r(\omega))$ lies in
$\mathscr{L}^1(\Omega,\FF,\mu;\RR)$.
\item
\label{a:2c}
There exists $r^*\in\GH$ such that the function
$\omega\mapsto\gw^*(r^*(\omega))$ lies in
$\mathscr{L}^1(\Omega,\FF,\mu;\RR)$.
\end{enumerate}
\end{assumption}

The following construct will also be required.

\begin{definition}[{\protect{\cite[Definition~1.4]{Cana24}}}]
Suppose that Assumption~\ref{a:1} is in force and, for every
$\omega\in\Omega$, let $\gw\colon\GW\to\RXX$. Suppose that, for
every $x\in\GH$, the function
$\Omega\to\RXX\colon\omega\mapsto\gw(x(\omega))$ is
$\FF$-measurable. The \emph{Hilbert direct integral} of the
functions $(\gw)_{\omega\in\Omega}$ relative to $\mathfrak{G}$ is
\begin{equation}
\leftindex^{\mathfrak{G}}{\int}_{\Omega}^{\oplus}
\gw\mu(d\omega)\colon\GG\to\RXX\colon x\mapsto
\int_{\Omega}\gw\brk1{x(\omega)}\mu(d\omega).
\end{equation}
\end{definition}

We introduce below parametrized versions of the integral proximal
mixtures of \cite[Definition~4.2]{Jota24}.

\begin{definition}
\label{d:ipm}
Suppose that Assumptions~\ref{a:3} and \ref{a:2} are in force, and
let $\gamma\in\RPP$. The \emph{integral proximal mixture} of
$(\gw)_{\omega\in\Omega}$ and $(\LW)_{\omega\in\Omega}$ with
parameter $\gamma$ is
\begin{equation}
\label{e:ipm}
\Rm{\gamma}(\LW,\gw)_{\omega\in\Omega}
=\mathsf{h}^*-\dfrac{1}{\gamma}\qq_{\HS},
\quad\text{where}\quad(\forall\mathsf{x}\in\HS)\quad
\mathsf{h}(\mathsf{x})=\int_{\Omega}
\moyo{\brk1{\gw^*}}{\frac{1}{\gamma}}
(\LW\mathsf{x})\mu(d\omega),
\end{equation}
and the \emph{integral proximal comixture} of
$(\gw)_{\omega\in\Omega}$ and $(\LW)_{\omega\in\Omega}$ with
parameter $\gamma$ is
\begin{equation}
\label{e:ipc}
\Rcm{\gamma}(\LW,\gw)_{\omega\in\Omega}
=\brk3{\Rm{1/\gamma}\brk1{\LW,\gw^*}_{\omega\in\Omega}}^*.
\end{equation}
\end{definition}

\subsection{Properties}
\label{sec:42}

The following proposition adopts the pattern of
\cite[Theorem~4.3]{Jota24} by connecting integral proximal mixtures
to proximal compositions in the more general context of
Definitions~\ref{d:1} and ~\ref{d:ipm}. 

\begin{proposition}
\label{p:60}
Suppose that Assumptions~\ref{a:3} and \ref{a:2} are in force, and
let $\gamma\in\RPP$. Define
\begin{equation}
\label{e:60b}
L\colon\HS\to\GG\colon\mathsf{x}\mapsto
\mathfrak{e}_{\LS}\mathsf{x}
\end{equation}
and
\begin{equation}
\label{e:60a}
g=\leftindex^{\mathfrak{G}}{\int}_{\Omega}^{\oplus}
\gw^{**}\mu(d\omega).
\end{equation}
Then the following hold:
\begin{enumerate}
\item
\label{p:60i}
$L\in\BL(\HS,\GG)$ and $0<\|L\|\leq 1$.
\item
\label{p:60ii}
$L^*\colon\GG\to\HS\colon
x^*\mapsto\int_{\Omega}\LW^*(x^*(\omega))\mu(d\omega)$.
\item
\label{p:60iii}
$g\in\Gamma_0(\GG)$.
\item
\label{p:60iv}
$\Rm{\gamma}(\LW,\gw)_{\omega\in\Omega}=L\proxc{\gamma}g$.
\item
\label{p:60v}
$\Rcm{\gamma}(\LW,\gw)_{\omega\in\Omega}=L\proxcc{\gamma}g$.
\end{enumerate}
\end{proposition}
\begin{proof}
\ref{p:60i}: We deduce from \cite[Proposition~3.12(ii)]{Cana24}
and Assumption~\ref{a:3}\ref{a:3d} that $L\in\BL(\HS,\GG)$ and that
$0<\|L\|^2\leq\int_{\Omega}\norm{\LW}^2\mu(d\omega)\leq 1$.

\ref{p:60ii}: See \cite[Proposition~3.12(v)]{Cana24}.

To establish \ref{p:60iii}--\ref{p:60v}, set 
$\vartheta\colon\Omega\to\RR\colon\omega\mapsto
-\gw^{**}(r(\omega))$ and $(\forall\omega\in\Omega)$ $\fw=\gw^*$.
Let us show that $(\fw)_{\omega\in\Omega}$ satisfies the following:
\begin{enumerate}[label={\normalfont[\Alph*]'}]
\item
\label{a:5a}
For every $\omega\in\Omega$, $\fw\in\Gamma_0(\GW)$.
\item
\label{a:5b}
For every $x\in\GH$, the mapping
$\omega\mapsto\prox_{\mathsf{f}_{\omega}}(x(\omega))$ lies
in $\mathfrak{G}$.
\item
\label{a:5c}
The function $\omega\mapsto\mathsf{f}_{\omega}(r^*(\omega))$
lies in $\mathscr{L}^1(\Omega,\FF,\mu;\RR)$.
\item
\label{a:5d}
$\vartheta\in\mathscr{L}^1(\Omega,\FF,\mu;\RR)$ and,
for every $\omega\in\Omega$, 
$\fw\geq\scal{r(\omega)}{\cdot}_{\GW}+\vartheta(\omega)$.
\end{enumerate}
This will confirm that $(\fw)_{\omega\in\Omega}$ satisfies the
properties of \cite[Assumption~4.6]{Cana24}. First, it follows from
items \ref{a:2a} and \ref{a:2b} in Assumption~\ref{a:2} and from
Lemma~\ref{l:1}\ref{l:1iv} that \ref{a:5a} holds. Second,
Assumption~\ref{a:2}\ref{a:2d} implies that \ref{a:5b} holds, while
Assumption~\ref{a:2}\ref{a:2c} implies that \ref{a:5c} holds. Let
us now show that
$\vartheta\in\mathscr{L}^1(\Omega,\FF,\mu;\RR)$.
As in the proof of \cite[Theorem~4.7(ix)]{Cana24},
$-\vartheta$ is $\FF$-measurable. Further, by \eqref{e:jjm2}
and Lemma~\ref{l:1}\ref{l:1i},
\begin{equation}
(\forall\omega\in\Omega)\;\;
\scal{\cdot}{r^*(\omega)}_{\GW}-\gw^*\brk1{r^*(\omega)}\leq
\gw^{**}\leq \gw.
\end{equation}
Thus, we infer from Assumption~\ref{a:2}\ref{a:2b}--\ref{a:2c}
that $\gw^{**}$ is bounded by integrable functions, which
shows that 
\begin{equation}
\label{e:jo1}
\vartheta\in\mathscr{L}^1(\Omega,\FF,\mu;\RR).
\end{equation}
On the other hand, it follows from Lemma~\ref{l:1}\ref{l:1iii}
and \eqref{e:jjm2} that, for every $\omega\in\Omega$,
$\fw=\gw^{***}\geq\scal{r(\omega)}{\cdot}_{\GW}-\gw^{**}
(r(\omega))=\scal{r(\omega)}{\cdot}_{\GW}+\vartheta(\omega)$, which
provides \ref{a:5d}. Therefore $(\fw)_{\omega\in\Omega}$ satisfies
the conclusions of \cite[Theorem~4.7]{Cana24}. In particular,
\cite[Theorem~4.7(i)--(ii)]{Cana24} entail that
\begin{equation}
\label{e:60c}
f=\leftindex^{\mathfrak{G}}{\int}_{\Omega}^{\oplus}
\fw\mu(d\omega)
\end{equation}
is a well-defined function in $\Gamma_0(\GG)$ and from 
\cite[Theorem~4.7(ix)]{Cana24} and Lemma~\ref{l:8}\ref{l:8i-} that 
\begin{equation}
\label{e:jo2}
g=f^*\in\Gamma_0(\GG).
\end{equation}

\ref{p:60iii}: See \eqref{e:jo2}.

\ref{p:60iv}: 
By \cite[Theorem~4.7(viii)]{Cana24}, 
\begin{equation}
\label{e:60d}
\moyo{f}{\frac{1}{\gamma}}
=\leftindex^{\mathfrak{G}}{\int}_{\Omega}^{\oplus}
\moyo{\fw}{\frac{1}{\gamma}}\mu(d\omega).
\end{equation}
Further, by \ref{p:60iii} and Lemma~\ref{l:8}\ref{l:8i-}, $g^*=f$.
In turn, \eqref{e:60b} and \eqref{e:60d} imply that
\begin{equation}
\label{e:60e}
\moyo{\brk1{g^*}}{\frac{1}{\gamma}}\circ L\colon\HS\to\RR\colon
\mathsf{x}\mapsto
\int_{\Omega}\moyo{\brk1{\gw^*}}{\frac{1}{\gamma}}
\brk1{\LW\mathsf{x}}\mu(d\omega).
\end{equation}
In view of Definitions~\ref{d:1} and \ref{d:ipm}, the
assertion is proved.

\ref{p:60v}:
Let us show that $(\fw)_{\omega\in\Omega}$ fulfills the properties
of Assumption~\ref{a:2} by showing that the following hold:
\begin{enumerate}[label={\normalfont[\Alph*]''}]
\item
\label{a:6a}
For every $\omega\in\Omega$, $\fw\colon\GW\to\RX$ satisfies
$\cam\fw\neq\emp$.
\item
\label{a:6b}
For every $x^*\in\GH$, the mapping
$\omega\mapsto\prox_{\fw^*}x^*(\omega)$ lies in $\mathfrak{G}$.
\item
\label{a:6c}
The function $\omega\mapsto\fw(r^*(\omega))$ lies in
$\mathscr{L}^1(\Omega,\FF,\mu;\RR)$.
\item
\label{a:6d}
The function $\omega\mapsto\fw^*(r(\omega))$ lies in
$\mathscr{L}^1(\Omega,\FF,\mu;\RR)$.
\end{enumerate}
We first note that \ref{a:5a} and Lemma~\ref{l:8}\ref{l:8i--} imply
that \ref{a:6a} holds, and that
\ref{a:5c}$\Leftrightarrow$\ref{a:6c}. Additionally, it follows
from \eqref{e:jo1} that \ref{a:6d} holds. It remains to establish
\ref{a:6b}. Assumption~\ref{a:2}\ref{a:2d} asserts that, for every
$x^*\in\mathfrak{H}$, the mapping
$\omega\mapsto\prox_{\fw}x^*(\omega)$ lies in $\mathfrak{G}$.
Therefore, the inclusion $\mathfrak{H}\subset\mathfrak{G}$, 
Lemma~\ref{l:8}\ref{l:8ii+}, and the fact the
$\mathfrak{G}$ is a vector space imply that, for every
$x^*\in\mathfrak{H}$, the mapping
$\omega\mapsto\prox_{\fw^*}x^*(\omega)=x^*(\omega)-
\prox_{\fw}x^*(\omega)$ lies in $\mathfrak{G}$, which provides
\ref{a:6b}. Hence, we combine Definition~\ref{d:ipm}, the
application of \ref{p:60iv} to $(\fw)_{\omega\in\Omega}$,
\eqref{e:jo2}, Lemma~\ref{l:8}\ref{l:8i-}, and
Definition~\ref{d:1}, to obtain
\begin{equation}
\Rcm{\gamma}(\LW,\gw)_{\omega\in\Omega}
=\brk2{\Rm{1/\gamma}(\LW,\fw)_{\omega\in\Omega}}^*
=\brk2{L\proxc{1/\gamma}f}^*
=\brk2{L\proxc{1/\gamma}g^*}^*
=L\proxcc{\gamma}g,
\end{equation}
which completes the proof.
\end{proof}

Our main results on integral proximal mixtures are the following.

\begin{theorem}
\label{t:65}
Suppose that Assumptions~\ref{a:3} and \ref{a:2} are in force, and
let $\gamma\in\RPP$. Then the following hold:
\begin{enumerate}
\item
\label{t:65i}
$\Rm{\gamma}(\LW,\gw)_{\omega\in\Omega}\in\Gamma_0(\HS)$.
\item
\label{t:65ii}
$\Rcm{\gamma}(\LW,\gw)_{\omega\in\Omega}\in\Gamma_0(\HS)$.
\item
\label{t:65iii}
$(\Rcm{\gamma}(\LW,\gw)_{\omega\in\Omega})^*
=\Rm{1/\gamma}(\LW,\gw^*)_{\omega\in\Omega}$.
\item
\label{t:65iv}
$\Rm{\gamma}(\LW,\gw)_{\omega\in\Omega}
=(\Rcm{1/\gamma}(\LW,\gw^*)_{\omega\in\Omega})^*$.
\item
\label{t:65v}
Let $\mathsf{x}\in\HS$. Then
$\prox_{\gamma\Rm{\gamma}(\LW,\gw)_{\omega\in\Omega}}\mathsf{x}=
\displaystyle\int_{\Omega}\LW^*\brk{\prox_{\gamma\gw^{**}}
(\LW\mathsf{x})}\,\mu(d\omega)$. 
\item
\label{t:65vi}
Let $\mathsf{x}\in\HS$. Then
$\prox_{\gamma\Rcm{\gamma}(\LW,\gw)_{\omega\in\Omega}}\mathsf{x}=
\mathsf{x}-\displaystyle\int_{\Omega}\LW^*(\LW\mathsf{x}-
\prox_{\gamma\gw^{**}}(\LW\mathsf{x}))\,\mu(d\omega)$.
\item
\label{t:65vii}
Define $g$ as in \eqref{e:60a} and $L$ as in \eqref{e:60b}.
Then the following are satisfied:
\begin{enumerate}
\item
\label{t:65viia}
$\partial(\Rm{\gamma}(\LW,\gw)_{\omega\in\Omega})
=L^*\pushfwd(\partial g+(\Id_{\GG}-L\circ L^*)/\gamma)$.
\item
\label{t:65viib}
$\partial(\Rcm{\gamma}(\LW,\gw)_{\omega\in\Omega})
=L^*\circ(\partial g^*+\gamma(\Id_{\GG}-L\circ L^*))^{-1}\circ L$.
\end{enumerate}
\item
\label{t:65viii}
Let $\mathsf{x}\in\HS$. Then
$\moyo{(\Rcm{\gamma}(\LW,\gw)_{\omega\in\Omega})}{\gamma}
(\mathsf{x})=\displaystyle\int_{\Omega}
\moyo{(\gw^{**})}{\gamma}(\LW\mathsf{x})\,\mu(d\omega)$.
\item
\label{t:65ix}
$\Argmin_{\mathsf{x}\in\HS}(\Rcm{\gamma}
(\LW,\gw)_{\omega\in\Omega})(\mathsf{x})=
\Argmin_{\mathsf{x}\in\HS}\displaystyle\int_{\Omega}
\moyo{(\gw^{**})}{\gamma}(\LW\mathsf{x})\,\mu(d\omega)$.
\item
\label{t:65x}
Let $\mathsf{x}\in\HS$. Then
$(\rec\Rcm{\gamma}(\LW,\gw)_{\omega\in\Omega})(\mathsf{x})
=\displaystyle\int_{\Omega}(\rec(\gw^{**}))(\LW\mathsf{x})\,
\mu(d\omega)$.
\item
\label{t:65xi}
Suppose that $\mu$ is a probability measure and that there exists
$\beta\in\RPP$ such that, for every $\omega\in\Omega$, 
$\gw\colon\GW\to\RR$ is convex and $\beta$-Lipschitzian. Then
$\Rcm{\gamma}(\LW,\gw)_{\omega\in\Omega}$ is $\beta$-Lipschitzian.
\end{enumerate}
\end{theorem}
\begin{proof}
Define $L$ as in \eqref{e:60b} and $g$ as in \eqref{e:60a}. Recall
from items \ref{p:60i} and \ref{p:60iii} in Proposition~\ref{p:60}
that $L\in\BL(\HS,\GG)$, $0<\|L\|\leq 1$, and $g\in\Gamma_0(\GG)$.
Additionally, by Proposition~\ref{p:60}\ref{p:60iv}--\ref{p:60v},
\begin{equation}
\label{e:100}
\Rm{\gamma}(\LW,\gw)_{\omega\in\Omega}=L\proxc{\gamma}g
\quad\text{and}\quad
\Rcm{\gamma}(\LW,\gw)_{\omega\in\Omega}=L\proxcc{\gamma}g.
\end{equation}
Also, proceeding as in the proof of Proposition~\ref{p:60},
it can be shown that
\begin{equation}
\label{e:73}
\brk1{\gw^{**}}_{\omega\in\Omega}\;\text{satisfies the properties
of \cite[Assumption~4.6]{Cana24}}. 
\end{equation}
Thus, by \cite[Theorem~4.7(iv)]{Cana24},
\begin{equation}
\label{e:120}
(\forall x\in\GG)\quad
\brk1{\prox_{\gamma g}x}(\omega)=
\prox_{\gamma\gw^{**}}\brk1{x(\omega)}\;\;\text{for $\mu$-almost
every}\;\omega\in\Omega.
\end{equation}

\ref{t:65i}--\ref{t:65iv}: These are consequences of \eqref{e:100}
and Proposition~\ref{p:7}.

\ref{t:65v}: It follows from \eqref{e:100},
Propositions~\ref{p:17}\ref{p:17i} and
\ref{p:60}\ref{p:60ii}, and \eqref{e:120} that
\begin{equation}
\prox_{\gamma\Rm{\gamma}(\LW,\gw)_{\omega\in\Omega}}\mathsf{x}
=L^*\brk1{\prox_{\gamma g}(L\mathsf{x})}
=\int_{\Omega}\LW^*\brk2{\prox_{\gamma\gw^{**}}
\brk1{\LW\mathsf{x}}}\,\mu(d\omega).
\end{equation}

\ref{t:65vi}: It follows from \eqref{e:100}, 
Propositions~\ref{p:17}\ref{p:17ii} and \ref{p:60}\ref{p:60ii}, and
\eqref{e:120} that
\begin{align}
\prox_{\gamma\Rcm{\gamma}(\LW,\gw)_{\omega\in\Omega}}\mathsf{x}
&=\mathsf{x}-L^*\brk1{L\mathsf{x}-\prox_{\gamma g}(L\mathsf{x})}
\nonumber\\
&=\mathsf{x}-\int_{\Omega}\LW^*\brk2{\LW\mathsf{x}
-\prox_{\gamma\gw^{**}}\brk1{\LW\mathsf{x}}}\mu(d\omega).
\end{align}

\ref{t:65vii}: A consequence of \eqref{e:100} and
Proposition~\ref{p:9}.

\ref{t:65viii}: By \eqref{e:73} and
\cite[Theorem~4.7(viii)]{Cana24},
\begin{equation}
\label{e:415}
\moyo{g}{\gamma}=\leftindex^{\mathfrak{G}}{\int}_{\Omega}^{\oplus}
\moyo{\brk1{\gw^{**}}}{\gamma}\mu(d\omega).
\end{equation}
However, by Lemma~\ref{l:8}\ref{l:8i-}, $g=g^{**}$. Therefore,
\eqref{e:100}, Proposition~\ref{p:10}\ref{p:10ii} and \eqref{e:415}
yield
\begin{equation}
\moyo{\brk2{\Rcm{\gamma}(\LW,\gw)_{\omega\in\Omega}}}{\gamma}
(\mathsf{x})=\moyo{\brk1{L\proxcc{\gamma}g}}{\gamma}
(\mathsf{x})=\moyo{g}{\gamma}(L\mathsf{x})=
\int_{\Omega}\moyo{\brk1{\gw^{**}}}{\gamma}(\LW\mathsf{x})\,
\mu(d\omega).
\end{equation}

\ref{t:65ix}: The assertion is obtained by using successively
\eqref{e:100}, Corollary~\ref{c:argmin}, and \ref{t:65viii}.

\ref{t:65x}: By \eqref{e:73} and \cite[Theorem~4.7(x)]{Cana24},
\begin{align}
\label{e:416}
\rec g=\leftindex^{\mathfrak{G}}{\int}_{\Omega}^{\oplus}
\rec\brk1{\gw^{**}}\mu(d\omega)
\end{align}
However, by Lemma~\ref{l:8}\ref{l:8i-}, $g=g^{**}$. Hence, it
results from \eqref{e:100}, Proposition~\ref{p:13}, and
\eqref{e:416} that
\begin{equation}
\brk2{\rec\Rcm{\gamma}(\LW,\gw)_{\omega\in\Omega}}(\mathsf{x})
=\brk2{\rec\brk1{L\proxcc{\gamma}g}}(\mathsf{x})
=\brk1{\rec g}(L\mathsf{x})=
\int_{\Omega}\brk1{\rec(\gw^{**})}(\LW\mathsf{x})\,\mu(d \omega).
\end{equation}

\ref{t:65xi}: 
It follows from \eqref{e:60a}, Lemma~\ref{l:8}\ref{l:8i-}, and
Jensen's inequality (\cite[Proposition~9.24]{Livre1}) that
\begin{align}
(\forall x\in\GG)(\forall y\in\GG)\quad
\lvert g(x)-g(y)\rvert^2
&=\biggl\lvert\int_{\Omega}\brk2{\gw\brk1{x(\omega)}
-\gw\brk1{y(\omega)}}\mu(d\omega)\biggr\rvert^2\nonumber\\
&\leq\int_{\Omega}\bigl\lvert\gw\brk1{x(\omega)}
-\gw\brk1{y(\omega)}\bigr\rvert^2\,\mu(d\omega)\nonumber\\
&\leq\beta^2\int_{\Omega}\norm{x(\omega)-y(\omega)}_{\GW}^2
\mu(d\omega)\nonumber\\
&=\beta^2\norm{x-y}_{\GG}^2.
\end{align}
Therefore, $g$ is $\beta$-Lipschitzian, and the conclusion follows
from \eqref{e:100} and Corollary~\ref{c:19}.
\end{proof}

Our second batch of results focuses on approximation properties.

\begin{theorem}
\label{t:70}
Suppose that Assumptions~\ref{a:3} and \ref{a:2} are in force.
For every $\mathsf{x}\in\HS$, define
\begin{equation}
\brk2{\pcm(\LW,\gw)_{\omega\in\Omega}}(\mathsf{x})=
\inf\Menge3{\int_{\Omega}\gw^{**}(x(\omega))
\mu(d\omega)}{x\in\GG\,\,\text{and}\,\,
\int_{\Omega}\LW^*\brk{x(\omega)}\mu(d\omega)=\mathsf{x}}
\end{equation}
and write $(\pcm(\LW,\gw)_{\omega\in\Omega})(\mathsf{x})=
(\epcm(\LW,\gw)_{\omega\in\omega})(\mathsf{x})$ if the infimum is
attained. Then the following hold:
\begin{enumerate}
\item
\label{t:70i}
Let $\gamma\in\RPP$. Then 
$\Rm{\gamma}(\LW,\gw)_{\omega\in\Omega}\geq
\pcm(\LW,\gw)_{\omega\in\Omega}$.
\item
\label{t:70ii}
Let $\gamma\in\RPP$ and $\mathsf{x}\in\HS$. Then 
\begin{equation}
\displaystyle\int_{\Omega}\moyo{(\gw^{**})}
{\gamma}(\LW\mathsf{x})\,\mu(d\omega)
\leq\brk2{\Rcm{\gamma}(\LW,\gw)_{\omega\in\Omega}}(\mathsf{x})\leq
\displaystyle\int_{\Omega}\gw^{**}(\LW\mathsf{x})\,\mu(d\omega).
\end{equation}
\item
\label{t:70iii}
Let $\gamma\in\RPP$. Then 
$\Rcm{\gamma}(\LW,\gw)_{\omega\in\Omega}\leq
\Rm{\gamma}(\LW,\gw)_{\omega\in\Omega}$.
\item
\label{t:70iv}
Let $\gamma\in\RPP$ and suppose that $\mu$ is a probability measure
and that, for every $\omega\in\Omega$, $\LW$ is an isometry. Then 
$\Rm{\gamma}(\LW,\gw)_{\omega\in\Omega}
=\Rcm{\gamma}(\LW,\gw)_{\omega\in\Omega}$.
\item
Let $\gamma\in\RPP$ and suppose that $L$ in \eqref{e:60b} is a
coisometry. Then the following are satisfied:
\begin{enumerate}
\item
\label{t:70va}
$\Rm{\gamma}(\LW,\gw)_{\omega\in\Omega}
=\epcm(\LW,\gw)_{\omega\in\Omega}$.
\item
\label{t:70vb}
Let $\mathsf{x}\in\HS$. Then
$(\Rcm{\gamma}(\LW,\gw)_{\omega\in\Omega})(\mathsf{x})=
\displaystyle\int_{\Omega}\gw^{**}(\LW\mathsf{x})\,\mu(d\omega)$.
\end{enumerate}
\item
\label{t:70vi}
Let $\mathsf{x}\in\HS$. Then the following are satisfied:
\begin{enumerate}
\item
\label{t:70via}
$\displaystyle\lim_{\gamma\to\pinf}
(\Rm{\gamma}(\LW,\gw)_{\omega\in\Omega})(\mathsf{x})
=(\pcm(\LW,\gw)_{\omega\in\Omega})(\mathsf{x})$.
\item
\label{t:70vib}
$\displaystyle\lim_{0<\gamma\to 0}
(\Rcm{\gamma}(\LW,\gw)_{\omega\in\Omega})(\mathsf{x})
=\int_{\Omega}\gw^{**}(\LW\mathsf{x})\,\mu(d\omega)$.
\end{enumerate}
\item
\label{t:70vii}
Suppose that $\HS$ and $\GG$ are finite-dimensional, and let
$(\gamma_n)_{n\in\NN}$ be a sequence in $\RPP$.
Then the following are satisfied:
\begin{enumerate}
\item
\label{t:70viia}
Suppose that $\gamma_n\uparrow\pinf$. Then
$\Rm{\gamma_n}(\LW,\gw)_{\omega\in\Omega}\xrightarrow{e}
\brk2{\pcm(\LW,\gw)_{\omega\in\Omega}}^{{\mbox{\raisebox{-1mm}
{\large$\breve{}$}}}}$.
\item
\label{t:70viib}
Suppose that $\gamma_n\downarrow 0$. Then
$\Rcm{\gamma_n}(\LW,\gw)_{\omega\in\Omega}\xrightarrow{e}\fs,
\;\text{where}\;(\forall\mathsf{x}\in\HS)\;\fs(\mathsf{x})=
\displaystyle\int_{\Omega}\gw^{**}(\LW\mathsf{x})\,\mu(d\omega)$.
\item
\label{t:70viic}
Suppose that $\gamma_n\downarrow 0$ and that the function
$\mathsf{x}\mapsto\displaystyle\int_{\Omega}
\gw^{**}(\LW\mathsf{x})\mu(d\omega)$ is proper and coercive. Then
$\inf_{\mathsf{x}\in\HS}(\Rcm{\gamma_n}(\LW,\gw)_{\omega\in\Omega})
(\mathsf{x})\to\min_{\mathsf{x}\in\HS}\displaystyle\int_{\Omega}
\gw^{**}(\LW\mathsf{x})\,\mu(d\omega)$.
\end{enumerate}
\end{enumerate}
\end{theorem}
\begin{proof}
Define $L$ as in \eqref{e:60b} and $g$ as in \eqref{e:60a}, and
recall from items \ref{p:60i} and \ref{p:60iii} of
Proposition~\ref{p:60} that $L\in\BL(\HS,\GG)$, $0<\|L\|\leq 1$,
and $g\in\Gamma_0(\GG)$. Further, by
Proposition~\ref{p:60}\ref{p:60iv}--\ref{p:60v},
\begin{equation}
\label{e:70c}
\Rm{\gamma}(\LW,\gw)_{\omega\in\Omega}=L\proxc{\gamma}g,
\;\;\text{and}\;\;
\Rcm{\gamma}(\LW,\gw)_{\omega\in\Omega}=L\proxcc{\gamma}g.
\end{equation}
Additionally, Proposition~\ref{p:60}\ref{p:60ii} yields
\begin{equation}
\label{e:70a}
(\forall\mathsf{x}\in\HS)\quad
\brk1{L^*\pushfwd g}(\mathsf{x})=
\inf_{\substack{x\in\GG\\L^*x=\mathsf{x}}}g(x)
=\brk2{\pcm(\LW,\gw)_{\omega\in\Omega}}(\mathsf{x}).
\end{equation}
On the other hand,
\begin{equation}
\label{e:70b}
(\forall\mathsf{x}\in\HS)\quad
g(L\mathsf{x})
=\int_{\Omega}\gw^{**}\brk1{\brk{\mathfrak{e}_{\LS}\mathsf{x}}
(\omega)}\mu(d\omega)
=\int_{\Omega}\gw^{**}\brk1{\LW\mathsf{x}}\mu(d\omega).
\end{equation}

\ref{t:70i}: 
The assertion follows from \eqref{e:70c}, \eqref{e:70a}, and
Proposition~\ref{p:20}\ref{p:20i}.

\ref{t:70ii}: Combine \eqref{e:70c}, \eqref{e:70b}, and
Proposition~\ref{p:20}\ref{p:20ii}.

\ref{t:70iii}: This is a consequence of \eqref{e:70c} and
Proposition~\ref{p:20}\ref{p:20iii}.

\ref{t:70iv}: We have
\begin{equation}
(\forall\mathsf{x}\in\HS)\quad
\norm{L\mathsf{x}}_{\GG}^2
=\int_{\Omega}\norm{\LW\mathsf{x}}_{\GW}^2\mu(d\omega)
=\int_{\Omega}\norm{\mathsf{x}}_{\HS}^2\mu(d\omega)
=\mu(\Omega)\norm{\mathsf{x}}_{\HS}^2
=\norm{\mathsf{x}}_{\HS}^2.
\end{equation}
Therefore, $L$ is an isometry and the assertion follows from
\eqref{e:70c} and Proposition~\ref{p:20}\ref{p:20iv}.

\ref{t:70va}: This follows from \eqref{e:70c}, \eqref{e:70a}, and 
Proposition~\ref{p:20}\ref{p:20v}.

\ref{t:70vb}: This follows from \eqref{e:70c}, \eqref{e:70b}, and
Proposition~\ref{p:20}\ref{p:20v}.

\ref{t:70via}: This follows from \eqref{e:70c}, \eqref{e:70a}, and
Theorem~\ref{t:45}\ref{t:45iii}.

\ref{t:70vib}: This follows from \eqref{e:70c}, \eqref{e:70b}, and
Theorem~\ref{t:45}\ref{t:45iv}.

\ref{t:70viia}: This follows from \eqref{e:70c},
\eqref{e:70a}, and Theorem~\ref{t:50}\ref{t:50iia}.

\ref{t:70viib}: This follows from \eqref{e:70c}, 
\eqref{e:70b}, and Theorem~\ref{t:50}\ref{t:50iib}.

\ref{t:70viic}: This follows from \eqref{e:70c}, 
\eqref{e:70b}, and Proposition~\ref{p:55}\ref{p:55i}.
\end{proof}

\begin{example}
\label{ex:12}
Let $p\in\NN\smallsetminus\{0\}$,
let $(\alpha_k)_{1\leq k\leq p}$ be a family in $\RPP$, let $\HS$
and $(\GK)_{1\leq k\leq p}$ be separable real Hilbert spaces,
let $\mathfrak{G}=\GS_1\times\cdots\times\GS_p$ be the usual
Cartesian product vector space, with generic element
$x=(\mathsf{x}_k)_{1\leq k\leq p}$, and, for every
$k\in\{1,\ldots,p\}$, let $\LS_k\in\BL(\HS,\GS_k)$ and let
$\gs_k\in\Gamma_0(\GS_k)$. Suppose that
$0<\sum_{k=1}^p\alpha_k\|\LS_k\|^2\leq 1$ and set
\begin{equation}
\label{e:ex12}
\Omega=\set{1,\ldots,p},\quad
\FF=2^{\set{1,\ldots,p}},\quad\text{and}\quad
\brk1{\forall k\in\set{1,\ldots,p}}\;\;
\mu\brk1{\set{k}}=\alpha_k.
\end{equation}
Then $((\GS_k)_{1\leq k\leq p},\mathfrak{G})$ is an
$\FF$-measurable vector field of real Hilbert spaces and the space
$\leftindex^{\mathfrak{G}}{\int}_{\Omega}^{\oplus}\GW\mu(d\omega)$
is the weighted Hilbert direct sum of $(\GS_k)_{1\leq k\leq p}$,
namely the Hilbert space obtained by equipping $\mathfrak{G}$ with
the scalar product $(x,y)\mapsto
\sum_{k=1}^p\alpha_k\scal{\mathsf{x}_k}{\mathsf{y}_k}_{\GS_k}$. 
Further, $\int_{\Omega}\|\LW\|^2\mu(d\omega)
=\sum_{k=1}^p\alpha_k\|\LS_k\|^2\in\rzeroun$. Therefore,
Assumptions~\ref{a:3} and \ref{a:2} are
satisfied, and \eqref{e:ipm} becomes a parametrized version of
the \emph{proximal mixture} of \cite[Example~5.9]{Svva23}, namely,
\begin{equation}
\label{e:rmi}
\Rm{\gamma}(\LS_k,\gs_k)_{1\leq k\leq p}=
\brk3{\sum_{k=1}^p\alpha_k\moyo{\brk1{\gs_k^*}}{\frac{1}{\gamma}}
\circ\LS_k}^*-\dfrac{1}{\gamma}\qq_{\HS},
\end{equation}
while \eqref{e:ipc} becomes a parametrized version of the 
\emph{proximal comixture} 
\begin{equation}
\label{e:rcm}
\Rcm{\gamma}(\LS_k,\gs_k)_{1\leq k\leq p}=
\brk4{\brk3{\sum_{k=1}^p\alpha_k\moyo{\brk1{\gs_k^{**}}}{\gamma}
\circ\LS_k}^*-\gamma\qq_{\HS}}^*.
\end{equation}
In particular, for every $\mathsf{x}\in\HS$, we derive from 
Theorem~\ref{t:70}\ref{t:70vi} the following new
facts:
\begin{enumerate}
\item
\label{ex:12i}
$\displaystyle\lim_{\gamma\to\pinf}
\brk2{\Rm{\gamma}(\LS_k,\gs_k)_{1\leq k\leq p}}(\mathsf{x})
=\brk2{\pcm(\LS_k,\gs_k)_{1\leq k\leq p}}(\mathsf{x})
=\inf_{\substack{\mathsf{y}_1\in\GS_1,\ldots,\mathsf{y}_p\in\GS_p\\
\sum_{k=1}^p\alpha_k\LS_k^*\mathsf{y}_k=\mathsf{x}}}
\brk3{\sum_{k=1}^p\alpha_k\gs_k^{**}(\mathsf{y}_k)}$.
\item
\label{ex:12ii}
$\displaystyle{\lim_{0<\gamma\to 0}
\brk2{\Rcm{\gamma}(\LS_k,\gs_k)_{1\leq k\leq p}}(\mathsf{x})
=\sum_{k=1}^p\alpha_k\gs_k^{**}(\LS_k\mathsf{x})}$.
\end{enumerate}
\end{example}

\begin{example}
\label{ex:13}
In the context of Example~\ref{ex:12}, suppose that $\HS$ is 
finite-dimensional and that, for every $k\in\{1,\ldots,p\}$,
$\GS_k$ is finite-dimensional and $\gs_k\in\Gamma_0(\GS_k)$.
Let $(\gamma_n)_{n\in\NN}$ be a sequence in $\RPP$. Then we obtain
the following new results on proximal mixtures and comixtures.
\begin{enumerate}
\item
\label{ex:13i}
Suppose that $\gamma_n\uparrow\pinf$. Then
Theorem~\ref{t:70}\ref{t:70viia} implies that
\begin{equation}
\Rm{\gamma_n}(\LS_k,\gs_k)_{1\leq k\leq p}\xrightarrow{e}
\brk2{\pcm(\LS_k,\gs_k)_{1\leq k\leq p}}^{{\mbox{\raisebox{-1mm}
{\large$\breve{}$}}}}.
\end{equation}
\item
\label{ex:13ii}
Suppose that $\gamma_n\downarrow 0$. Then
Theorem~\ref{t:70}\ref{t:70viib} implies that
$\Rcm{\gamma_n}(\LS_k,\gs_k)_{1\leq k\leq p}\xrightarrow{e}
\sum_{k=1}^p\alpha_k\gs_k\circ\LS_k$.
\item
\label{ex:13iii}
Suppose that $\gamma_n\downarrow 0$ and that the function
$\sum_{k=1}^p\alpha_k\gs_k\circ\LS_k$ is proper and coercive. Then 
Theorem~\ref{t:70}\ref{t:70viic} implies that
\begin{equation}
\inf_{\mathsf{x}\in\HS}\brk2{\Rcm{\gamma_n}
(\LS_k,\gs_k)_{1\leq k\leq p}}(\mathsf{x})\to
\min_{\mathsf{x}\in\HS}\sum_{k=1}^p\alpha_k\gs_k(\LS_k\mathsf{x}).
\end{equation}
\end{enumerate}
\end{example}

\begin{remark}
\label{r:33}
In connection with Example~\ref{ex:13}, it was empirically argued 
in \cite{Eusi24} (see also \cite{Huzz22,Liti20,Shen17,Yuyl13} for
the special cases of proximal averages) that, in variational
formulations arising in image recovery and machine learning,
combining linear operators $(\LS_k)_{1\leq k\leq p}$ and convex
functions $(\mathsf{g}_k)_{1\leq k\leq p}$ by means of the proximal
comixture \eqref{e:rcm} instead of the standard averaging operation
$\sum_{k=1}^p\alpha_k\mathsf{g}_k\circ\LS_k$ had modeling and
numerical advantages. For instance, the proximity of the former is
intractable in general \cite{Jmaa11}, while that of the latter is
explicitly given by Theorem~\ref{t:65}\ref{t:65vi} to be
$\IdHS-\sum_{k=1}^p\alpha_k\brk{\LS_k^*\circ
(\mathsf{Id}_{\GS_k}-\prox_{\gamma\mathsf{g}_k})\circ\LS_k}$, which
makes the implementation of first-order optimization algorithms
\cite{Acnu24} straightforward. The results of Example~\ref{ex:13}
provide a theoretical context that sheds more light on such an
approximation.
\end{remark}

\subsection{Proximal expectations}
\label{sec:43}

We specialize the results of Section~\ref{sec:42} to the proximal
expectation. This operation, introduced in
\cite[Definition~4.6]{Jota24} as an extension of the proximal
average for finite families, performs a nonlinear averaging of
an arbitrary family of functions. We study here the following 
extension of it which incorporates a parameter.

\begin{definition}
\label{d:201}
Let $(\Omega,\FF,\PP)$ be a complete probability space, let $\HS$
be a separable real Hilbert space, let $(\fw)_{\omega\in\Omega}$ be
a family of functions in $\Gamma_0(\HS)$ such that the function
\begin{equation}
\Omega\times\HS\to\RX\colon(\omega,\mathsf{x})\mapsto
\fw(\mathsf{x})
\end{equation}
is $\FF\otimes\BE_{\HS}$-measurable. Suppose that there exist
vectors $r\in\mathscr{L}^2\brk{\Omega,\FF,\PP;\HS}$ and
$r^*\in\mathscr{L}^2\brk{\Omega,\FF,\PP;\HS}$ such that the
functions $\omega\mapsto\fw(r(\omega))$ and
$\omega\mapsto\fw^*(r^*(\omega))$ belong to
$\mathscr{L}^1(\Omega,\FF,\PP;\RR)$. The \emph{proximal
expectation} of 
$(\fw)_{\omega\in\Omega}$ with parameter $\gamma\in\RPP$ is
\begin{equation}
\label{e:fr2}
\PE_{\gamma}(\fw)_{\omega\in\Omega}
=\mathsf{h}^*-\dfrac{1}{\gamma}\qq_{\HS},
\quad\text{where}\quad(\forall\mathsf{x}\in\HS)\quad
\mathsf{h}(\mathsf{x})=\int_{\Omega}
\moyo{\brk1{\fw^*}}{\frac{1}{\gamma}}
(\mathsf{x})\PP(d\omega).
\end{equation}
\end{definition}

An inspection of Definition~\ref{d:ipm} suggests that the proximal
expectation can be viewed as the instance of the integral proximal
mixture in which $(\forall\omega\in\Omega)$ $\GW=\HS$ and
$\LW=\IdHS$. This fact opens the possibility of specializing the 
results of Section~\ref{sec:42} to obtain properties of the
proximal expectation. Let us formalize these ideas.

\begin{proposition}
\label{p:79}
Consider the setting of Definition~\ref{d:201} and let
$\gamma\in\RPP$. Then the following hold:
\begin{enumerate}
\item
\label{p:79i}
$\PE_{\gamma}\brk{\fw}_{\omega\in\Omega}
=\Rm{\gamma}\brk1{\IdHS,\fw}_{\omega\in\Omega}
=\Rcm{\gamma}\brk1{\IdHS,\fw}_{\omega\in\Omega}$.
\item
\label{p:79ii}
$\PE_{\gamma}\brk{\fw}_{\omega\in\Omega}\in\Gamma_0(\HS)$.
\item
\label{p:79iii}
$(\PE_{\gamma}\brk{\fw}_{\omega\in\Omega})^*
=\PE_{1/\gamma}\brk{\fw^*}_{\omega\in\Omega}$.
\item
\label{p:79iv}
Let $\mathsf{x}\in\HS$. Then
$\prox_{\gamma\PE_{\gamma}(\fw)_{\omega\in\Omega}}\mathsf{x}
=\displaystyle\int_{\Omega}\prox_{\gamma\fw}\mathsf{x}
\,\PP(d\omega)$.
\item
\label{p:79vi}
Let $\mathsf{x}\in\HS$. Then
$\moyo{(\PE_{\gamma}(\fw)_{\omega\in\Omega})}{\gamma}(\mathsf{x})
=\displaystyle\int_{\Omega}\moyo{\fw}{\gamma}(\mathsf{x})
\,\PP(d\omega)$.
\item
\label{p:79vii}
$\Argmin_{\mathsf{x}\in\HS}(\PE_{\gamma}(\fw)_{\omega\in\Omega})
(\mathsf{x})=
\Argmin_{\mathsf{x}\in\HS}\displaystyle\int_{\Omega}
\moyo{\fw}{\gamma}(\mathsf{x})\,\PP(d\omega)$.
\item
\label{p:79viii}
Let $\mathsf{x}\in\HS$. Then
$(\rec\PE_{\gamma}(\fw)_{\omega\in\Omega})(\mathsf{x})
=\displaystyle\int_{\Omega}(\rec\fw)(\mathsf{x})\,\PP(d\omega)$.
\item
\label{p:79ix}
Suppose that there exists $\beta\in\RPP$ such that, for every
$\omega\in\Omega$, $\fw\colon\HS\to\RR$ is $\beta$-Lipschitzian.
Then $\PE_{\gamma}(\fw)_{\omega\in\Omega}$ is $\beta$-Lipschitzian.
\end{enumerate}
\end{proposition}
\begin{proof}
\ref{p:79i}: As in the proof of \cite[Proposition~4.7]{Jota24},
the family $(\fw)_{\omega\in\Omega}$ fulfills the properties of
Assumption~\ref{a:2}. Therefore, the conclusion follows from
\eqref{e:fr2}, \eqref{e:ipm}, and Theorem~\ref{t:70}\ref{t:70iv}.

\ref{p:79ii}--\ref{p:79ix}: Combine \ref{p:79i} and
Theorem~\ref{t:65}.
\end{proof}

\begin{remark}
Item \ref{p:79iv} in Proposition~\ref{p:79} justifies calling
$\PE_{\gamma}(\fw)_{\omega\in\Omega}$ the proximal expectation of 
$(\fw)_{\omega\in\Omega}$: its proximity operator is the expected
value of the individual ones. 
\end{remark}

\begin{proposition}
\label{p:80}
Consider the setting of Definition~\ref{d:201}. For every
$\mathsf{x}\in\HS$, define
\begin{equation}
\brk2{\pex(\fw)_{\omega\in\Omega}}(\mathsf{x})=
\inf\Menge3{\int_{\Omega}\fw(x(\omega))
\PP(d\omega)}{x\in L^2(\Omega,\FF,\PP;\HS)\,\,\text{and}\,\,
\int_{\Omega}x(\omega)\PP(d\omega)=\mathsf{x}}.
\end{equation}
Then the following hold:
\begin{enumerate}
\item
\label{p:80i}
Let $\gamma\in\RPP$ and $\mathsf{x}\in\HS$. Then
$(\PE_{\gamma}(\fw)_{\omega\in\Omega})(\mathsf{x})\geq
\displaystyle\int_{\Omega}\moyo{\fw}{\gamma}(\mathsf{x})
\,\PP(d\omega)$.
\item
\label{p:80ii}
Let $\gamma\in\RPP$ and $\mathsf{x}\in\HS$. Then
\begin{equation}
\brk2{\pex(\fw)_{\omega\in\Omega}}(\mathsf{x})
\leq\brk2{\PE_{\gamma}(\fw)_{\omega\in\Omega}}(\mathsf{x})
\leq\displaystyle\int_{\Omega}\fw(\mathsf{x})\,\PP(d\omega).
\end{equation}
\item
\label{p:80iii}
Let $\mathsf{x}\in\HS$. Then the following are satisfied:
\begin{enumerate}
\item
\label{p:80iiia}
$\displaystyle\lim_{\gamma\to\pinf}
(\PE_{\gamma}(\fw)_{\omega\in\Omega})(\mathsf{x})
=(\pex(\fw)_{\omega\in\Omega})(\mathsf{x})$.
\item
\label{p:80iiib}
$\displaystyle\lim_{0<\gamma\to 0}
(\PE_{\gamma}(\fw)_{\omega\in\Omega})(\mathsf{x})
=\int_{\Omega}\fw(\mathsf{x})\,\PP(d\omega)$.
\end{enumerate}
\item
\label{p:80iv}
Suppose that $\HS$ and $\GG$ are finite-dimensional, and let
$(\gamma_n)_{n\in\NN}$ be a sequence in $\RPP$.
Then the following are satisfied:
\begin{enumerate}
\item
\label{p:80iva}
Suppose that $\gamma_n\uparrow\pinf$. Then
$\PE_{\gamma_n}(\fw)_{\omega\in\Omega}\xrightarrow{e}
\brk1{\pex(\fw)_{\omega\in\Omega}}^{{\mbox{\raisebox{-1mm}
{\large$\breve{}$}}}}$.
\item
\label{p:80ivb}
Suppose that $\gamma_n\downarrow 0$. Then
$\PE_{\gamma_n}(\fw)_{\omega\in\Omega}\xrightarrow{e}\fs$,
where $\fs\colon\mathsf{x}\mapsto
\displaystyle\int_{\Omega}\fw(\mathsf{x})\,\PP(d\omega)$.
\item
\label{p:80ivc}
Suppose that $\gamma_n\downarrow 0$ and that the function
$\mathsf{x}\mapsto\displaystyle\int_{\Omega}
\fw(\mathsf{x})\PP(d\omega)$ is proper and coercive. Then
$\inf_{\mathsf{x}\in\HS}(\PE_{\gamma_n}(\fw)_{\omega\in\Omega})
(\mathsf{x})\to\min_{\mathsf{x}\in\HS}
\displaystyle\int_{\Omega}\fw(\mathsf{x})\PP(d\omega)$.
\end{enumerate}
\end{enumerate}
\end{proposition}
\begin{proof}
Combine Proposition~\ref{p:79}\ref{p:79i} and Theorem~\ref{t:70}.
\end{proof}

\begin{remark}
\label{r:80}
Suppose that $(\fs_k)_{1\leq k\leq p}$ is a finite family of
functions in $\Gamma_0(\HS)$ and define $\PP$ as in
\eqref{e:ex12},
with the additional assumption that $\sum_{k=1}^p\alpha_k=1$. Then 
$\PE(\fs_k)_{1\leq k\leq p}$ is the
\emph{proximal average} of $(\fs_k)_{1\leq k\leq p}$, studied in
\cite{Baus08} (see also \cite[Example~5.9]{Svva23}), namely,
\begin{equation}
\PE_{\gamma}(\fs_k)_{1\leq k\leq p}=\brk3{\sum_{k=1}^p\alpha_k
\moyo{\brk1{\fs_k^*}}{\frac{1}{\gamma}}}^*
-\dfrac{1}{\gamma}\qq_{\HS}=\pav_{\gamma}(\fs_k)_{1\leq k\leq p}.
\end{equation}
In this context, Propositions~\ref{p:79}\ref{p:79i}--\ref{p:79vii}
and \ref{p:80} recover properties presented in \cite{Baus08}. On
the other hand, Proposition~\ref{p:79}\ref{p:79viii}--\ref{p:79ix}
yields the following new properties of the proximal average:
\begin{enumerate}
\item
\label{r:80iii}
$\rec(\pav_{\gamma}(\fs_k)_{1\leq k\leq p})
=\sum_{k=1}^p\alpha_k\rec\fs_k$.
\item
\label{r:80iv}
Suppose that there exists $\beta\in\RPP$ such that, for every
$k\in\{1,\ldots,p\}$, $\fs_k\colon\HS\to\RR$ is
$\beta$-Lipschitzian. Then $\pav_{\gamma}(\fs_k)_{1\leq k\leq p}$
is $\beta$-Lipschitzian.
\end{enumerate}
\end{remark}

We conclude by making a connection between proximal expectations
and integral proximal comixtures that extends
Proposition~\ref{p:79}\ref{p:79i}.

\begin{proposition}
\label{p:75}
Let $(\Omega,\FF,\PP)$ be a complete probability space, suppose
that Assumptions~\ref{a:3} and \ref{a:2} are in force, and let
$\gamma\in\RPP$. Further, for every $\omega\in\Omega$, suppose that
$0<\|\LW\|\leq 1$ and set 
$\fw=\LW\proxcc{\gamma}\gw$. Suppose that the function
\begin{equation}
\Omega\times\HS\to\RX\colon(\omega,\mathsf{x})\mapsto
\fw(\mathsf{x})
\end{equation}
is $\FF\otimes\BE_{\HS}$-measurable and that there exist
$s\in\mathscr{L}^2\brk{\Omega,\FF,\PP;\HS}$ and
$s^*\in\mathscr{L}^2\brk{\Omega,\FF,\PP;\HS}$ such that the
functions
$\omega\mapsto\fw(s(\omega))$ and
$\omega\mapsto\fw^*(s^*(\omega))$ lie in
$\mathscr{L}^1(\Omega,\FF,\PP;\RR)$. Then
\begin{equation}
\label{e:p75a}
\PE_{\gamma}\brk2{\LW\proxcc{\gamma}\gw}_{\omega\in\Omega}
=\Rcm{\gamma}(\LW,\gw)_{\omega\in\Omega}.
\end{equation}
\end{proposition}
\begin{proof}
As in the proof of \cite[Proposition~4.7]{Jota24}, the family
$(\fw)_{\omega\in\Omega}$ fulfills the properties of
Assumption~\ref{a:2}. On the other hand,
Proposition~\ref{p:79}\ref{p:79ii} and
Theorem~\ref{t:65}\ref{t:65ii} assert that 
$\PE_{\gamma}(\fw)_{\omega\in\Omega}$ and
$\Rcm{\gamma}(\LW,\gw)_{\omega\in\Omega}$ are in $\Gamma_0(\HS)$.
Further, Propositions~\ref{p:79}\ref{p:79vi} and
\ref{p:10}\ref{p:10ii}, together with
Theorem~\ref{t:65}\ref{t:65viii} yield
\begin{align}
\label{e:p75b}
(\forall\mathsf{x}\in\HS)\quad
\moyo{\brk2{\PE_{\gamma}(\fw)_{\omega\in\Omega}}}{\gamma}
(\mathsf{x})
&=\int_{\Omega}\moyo{\fw}{\gamma}(\mathsf{x})\,
\PP(d\omega)\nonumber\\
&=\int_{\Omega}\moyo{\brk1{\LW\proxcc{\gamma}\gw}}{\gamma}
(\mathsf{x})\,\PP(d\omega)\nonumber\\
&=\int_{\Omega}\moyo{\brk1{\gw^{**}}}{\gamma}(\LW\mathsf{x})\,
\PP(d\omega)\nonumber\\
&=\moyo{\brk2{\Rcm{\gamma}(\LW,\gw)_{\omega\in\Omega}}}{\gamma}
(\mathsf{x}),
\end{align}
and the assertion therefore follows from Lemma~\ref{l:6}.
\end{proof}

\end{document}